\documentclass[11pt]{amsart}

\usepackage{amssymb}
\usepackage{graphicx}
\usepackage[all]{xy}

\usepackage{tkz-euclide}
\usepackage{tikz}

\newtheorem{thm}{Theorem}[section]

\newtheorem{prop}[thm]{Proposition}
\newtheorem{lem}[thm]{Lemma}

\newtheorem{rem}[thm]{Remark}

\newtheorem{definition}[thm]{Definition}

\newcommand{\C}{\mathbb C}

\newcommand{\R}{\mathbb R}
\newcommand{\HH}{{\mathbb H}^2}

\renewcommand{\leq}{\leqslant}
\renewcommand{\geq}{\geqslant}
\renewcommand{\epsilon}{\varepsilon}
\renewcommand{\phi}{\varphi}

\title{Bounded ideal triangulations of infinite Riemann surfaces}

\author{Dragomir \v Sari\' c}
\address{DS: Mathematics PhD. Program, Graduate Center of the City University of New York, 365 Fifth Avenue, New York, NY 10016-4309}
\address{DS: Department of Mathematics, Queens College of the City University of New York, 65--30 Kissena Blvd., Flushing, NY 11367}
\email{Dragomir.Saric@qc.cuny.edu}

\author{Casey Whitney}
\address{CW:  Mathematics PhD. Program, Graduate Center of the City University of New York, 365 Fifth Avenue, New York, NY 10016-4309}
\email{}
\thanks{}

\thanks{The second author was partially supported by a PSC-CUNY research grant.}

\date{\today}

\subjclass{}
\keywords{}

\begin{document}
	
\begin{abstract}
We introduce a notion of  a bounded ideal triangulation of an infinite Riemann surface  and  parametrize Teichm\"uller spaces of infinite surfaces which allow bounded triangulations. We prove that our parametrization is real-analytic. Riemann surfaces with bounded geometry and countably many punctures belong to the class of surfaces with bounded ideal triangulations. In comparison, the Fenchel-Nielsen parametrization for surfaces with bounded geometry is not known, while the Fenchel-Nielsen parametrization for surfaces with bounded pants decompositions is known as a homeomorphism but it is not known whether it is real-analytic
\end{abstract}

\maketitle

\section{Introduction}

Any marked Riemann surface structure on a compact topological surface $S$ of genus at least two is uniquely determined by the length and the twist parameters on each cuff of the pants decomposition-defining the Fenchel-Nielsen parametrization of the Teichm\"uller space $T(S)$. The Fenchel-Nielsen parametrization of $T(S)$ is real-analytic. 

Shiga \cite{Shiga} introduced a (topological) parametrization of the Teichm\"uller space $T(X)$ of an infinite Riemann surface $X$ using the Fenchel-Nielsen coordinates in the case when $X$ has a bounded pants decomposition. This parametrization of the Teichm\"uller space $T(X)$ was extended to the case when the surface $X$ has an upper bounded geodesic pants decomposition with a sequence of cuff lengths going to zero by Alessandrini, Liu, Papadopoulos, and Su \cite{ALPS}. 
It turns out that when a Riemann surface has only unbounded geodesic pants decomposition, the exact Fenchel-Nielsen coordinates of the Teichm\"uller space are not known. This has to do with the fact that cuffs with large lengths have thin collars that can bunch together. 

Arguably, the most natural class of infinite Riemann surfaces are those with bounded geometry.  While Riemann surfaces with bounded pants decompositions are of bounded geometry,  Kinjo \cite{Kinjo} and also Basmajian, Parlier, and Vlamis \cite{BPV} gave examples of surfaces with bounded geometry but no bounded geodesic pants decompositions.

Thurston \cite{Thurston}, Penner \cite{Penner} and Bonahon \cite{Bonahon} introduced another way of parametrizing the Teichm\"uller space of compact (or finite area) Riemann surfaces by shears on a fixed ideal geodesic triangulation of the surface. McLeay and Parlier \cite{MP} proved that any infinite Riemann surface has an ideal triangulation. The shear parametrization was first considered for the universal Teichm\"uller space by Penner \cite{Penner}, and subsequently in \cite{Saric10,Saric21,PS,SWW}. In this paper, we study the shear parametrization of Teichm\"uller spaces of infinite Riemann surfaces.

An ideal geodesic triangulation of an infinite Riemann surface $X$ is {\it bounded} if it is locally finite, every edge has both of its endpoints at punctures of $X$, a uniform constant bounds the number of incident edges at each puncture, and the shears have finite $\ell^{\infty}$-norm. We show that the Riemann surfaces in \cite{Kinjo,BPV} have bounded ideal geodesic triangulations (these surfaces are flute surfaces. i.e., planar surfaces with countably many punctures and a single topological end.) The surfaces in \cite{Kinjo} are not parabolic, while we prove that the surfaces in \cite{BPV} are parabolic. However, both classes of surfaces have bounded geometry without having a bounded geodesic pants decomposition. More generally, the class of surfaces $X$ with bounded ideal geodesic triangulations are not limited to planar surfaces or surfaces with only one topological end. Such surfaces can have an arbitrary genus (including infinite) and an arbitrary space of topological ends (including the Cantor set).

Let $G$ be the index two orientation-preserving subgroup of the $(2,4,8)$ hyperbolic triangle group acting on the upper half-plane $\mathbb{H}^2$, and let $P$ be the fundamental region for the action of $G$. By deleting vertices of $P$ along with three points in $P$ close to the vertices of $P$ and propagating the deletion under the group $G$ in $\mathbb{H}^2$, one obtains a new surface $X$ that is planar and has infinitely many punctures that accumulate to one infinite end (see \cite[Figure 1]{Kinjo}). Kinjo \cite{Kinjo} proved that any geodesic pants decomposition of $X$ has cuff lengths going to infinity. 
We establish 

\begin{prop}
\label{prop:triangulation-kinjo}
The Riemann surface $X$ obtained by erasing points in the upper half-plane $\mathbb{H}^2$ along the orbit of the index two orientation-preserving subgroup of the $(2,4,8)$ hyperbolic triangle group is not parabolic and has a bounded ideal triangulation. 
\end{prop}

The following example is a surface constructed by Basmajian, Parlier, and Vlamis \cite{BPV}. Start with the hyperbolic rectangle with three angles $\pi /2$ and the fourth angle $0$. The two finite sides are chosen to have equal lengths $\sinh^{-1} (1)$. Take countably many isometric rectangles as above and isometrically glue finite sides to finite sides such that four sides are at the vertex to make an angle $2\pi$, thus a regular point of the surface. At the vertex with one finite side and one infinite side of a rectangle, we again glue four rectangles to obtain an angle $2\pi$. Finally, at the ideal vertex, we also glue four rectangles to obtain a Riemann surface with countably many punctures accumulating to one infinite end. 

\begin{prop}
\label{prop:BPV-surface}
Let $X$ be a Riemann surface obtained by isometrically gluing infinitely many isometric rectangles with two finite sides of length $\sinh^{-1} (1)$, three angles $\pi /2$, and one angle $0$ as above. Then $X$ is parabolic, and there is an ideal triangulation of $X$ that is bounded and whose shears are all zero.
\end{prop}

As we pointed out, the surfaces in Propositions \ref{prop:triangulation-kinjo} and \ref{prop:BPV-surface} do not have even upper bounded geodesic pants decomposition, and their Teichm\"uller spaces are not (to our best knowledge) successfully parametrized by the Fenchel-Nielsen coordinates. Our main result is a parametrization of the Teichm\"uller spaces of surfaces with bounded ideal geodesic triangulation, which includes a large class of surfaces with bounded geometry and the above examples (see \S \ref{sec:planar} for more details). 

Our first result gives a homeomorphic parametrization of the Teichm\"uller space $T(X)$ in terms of shear functions on a bounded ideal hyperbolic triangulation $\mathcal{T}$ on $X$.

\begin{thm}
\label{thm:shears-homeo}
Let $X$ be a Riemann surface with a bounded ideal triangulation $\mathcal{T}$. Denote by $E(\mathcal{T})$ the edge set of $\mathcal{T}$. The shear map
$$
s_{*}:T(X)\to\ell^{\infty}(E(\mathcal{T}),\mathbb{R})
$$
is injective. The image $s_{*}(T(X))$ consists of all shear functions $g\in \ell^{\infty}(E(\mathcal{T}),\mathbb{R})$ such that the sum of $g$-values on each set of edges ending in a single puncture of $X$ is zero.

Moreover, the shear map $s_{*}$ is a homeomorphism onto its image for the natural topologies on $T(X)$ and $\ell^{\infty}(E(\mathcal{T}),\mathbb{R})$. In particular, $s_{*}(T(X))$ contains the constant zero function. The point in $T(X)$ corresponding to the zero shear function has a subgroup of $PSL_2(\mathbb{Z})$ for its covering group.
\end{thm}

When $X$ has a bounded or an upper bounded pants decomposition, the Fenchel-Nielsen coordinates give homeomorphic parametrizations in the corresponding space. Unlike for compact surfaces, it is not known whether these parametrizations are real-analytic  (although one would expect to be so).
We prove that the above shear parametrization is a real-analytic diffeomorphism onto its image.

\begin{thm}
\label{thm:shear-r-analytic}
Let $X$ be a Riemann surface with a bounded ideal hyperbolic triangulation $\mathcal{T}$. Then, the shear map
$$
s_{*}:T(X)\to\ell^{\infty} (E(\mathcal{T}),\mathbb{R})
$$
is a real-analytic diffeomorphism onto its image.
\end{thm}

We prove the above theorem by extending the shear map to a neighborhood of $T(X)$ in the quasi-Fuchsian space $QF(X)$. The extended shear map takes values in $\ell^{\infty}(E(\mathcal{T}),\mathbb{C})$ and it is defined by bending from the upper half-plane into upper half-space when the shears have non-zero imaginary parts. The bendings are done in terms of a single complex parameter, and we show that the bending induces a holomorphic motion of the real axis. Using the properties of holomorphic motions, we prove that the bending is holomorphic in a single complex parameter (for the origin of this idea, see \cite{KeenSeries,Saric,Saric14}). Since a map from one complex Banach space onto another is holomorphic if and only if it is bounded and weakly holomorphic, the above implies that the extended map is holomorphic (see \cite{Chae}). Hence, its restriction to $T(X)$ is real-analytic. The idea that we can use the holomorphicity of the bending in a single variable to prove that the map from $T(X)$ is holomorphic is new in this setup.

The critical step in proving that the bending map is holomorphic is to prove that it is injective on the real axis. This phenomenon results from specific cancellation properties of the composition of two loxodromic elements whose shears are opposite but of equal absolute value (see \cite{Bonahon,Saric14}). Since the sum of shears at a puncture is zero and there is a bound on the number of edges incident to each puncture, this phenomenon also induces appropriate cancellations in the case of infinite Riemann surfaces.

Finally, we consider the big mapping class group $MCG(X)$ and the quasiconformal mapping class group $QMCG(X)$, where $QMCG(X)<MCG(X)$.  Given $[f]\in MCG(X)$, we denote by $\mathcal{T}_{[f]}$ the ideal hyperbolic triangulation of $X$ obtained by straightening the edges of $f(E(\mathcal{T}))$ into hyperbolic geodesics via homotopy modulo the punctures. We will say that $\mathcal{T}$ and $\mathcal{T}_{[f]}$ have a {\it finite intersection property} if there is a constant $M$ such that each edge of $\mathcal{T}$ intersects at most $M$ edges of $\mathcal{T}_{[f]}$, and each edge of $\mathcal{T}_{[f]}$ intersects at most $M$ edges of $\mathcal{T}$. Using the fact that $X$ is quasiconformally equivalent to a Riemann surface on which the image of $\mathcal{T}$ has zero shears and a result of Parlier and the second author \cite{PS}, we obtain

 \begin{thm}
 \label{thm:qmcg}
 Let $X$ be a Riemann surface with a bounded ideal hyperbolic triangulation $\mathcal{T}$. Then, an element $[f]$ in the big mapping class group $MCG(X)$ belongs to the quasiconformal mapping class group $QMCG(X)$ if and only if the triangulations $\mathcal{T}$ and $\mathcal{T}_{[f]}$ have a finite intersection property.
 \end{thm}
 
 The above theorem gives a topological criteria for checking whether $[f]\in MCG(X)$ does not have a quasiconformal representative in its homotopy class on $X$. It is enough to establish that the intersection number is infinite.

\section{A class of planar Riemann surfaces with bounded geometry}
\label{sec:planar}

A {\it geodesic pair of pants} is a bordered Riemann surface (whose interior is homeomorphic to a sphere minus three disks) equipped with the conformal hyperbolic metric such that the boundary components are either closed geodesics (called cuffs) or punctures, where at least one boundary component is a cuff. We will say that a family of geodesic pairs of pants is {\it bounded} if there are two positive numbers such that the cuff lengths are between two numbers for the whole family of pants.

A Riemann surface is said to have {\it bounded geometry} if there is a lower and upper bound on the injectivity radius away from horodisk neighborhoods of punctures with the boundary of the length $2$ (see \cite{Kinjo,BPV}). If a surface has a bounded geometry, the covering group is of the first kind.
A surface $X$ is planar if and only if it is conformally equivalent to a domain in $\mathbb{C}$. 

When $X$ is a surface with bounded geometry, Kinjo \cite{Kinjo1} proved that it can be decomposed into at most countably many geodesic subsurface components $\{ X_i\}_i$ such that each component $X_i$ either has a geodesic pants decomposition with cuff lengths between two positive constants or it is a planar subsurface that can be decomposed into right-angled hexagons with side lengths bounded between two positive constants. A planar component has countably many closed geodesics on its boundary, and each hexagon has three non-consecutive sides on the boundary geodesics of the planar component. Each boundary geodesic of a planar component is the union of finitely many sides of hexagons. 
The component subsurfaces attached to the planar components are attached along simple closed geodesics. Necessarily, all the punctures are in the attached components to the planar components. 

Each attached component (to a planar component) can have a finite or infinite topology. The simplest topology of a single attached component is a pair of pants with one closed geodesic and two punctures. Generally, the single attached piece has a bounded pants decomposition, with pairs of pants possibly having one or two punctures as their respective boundaries. However, the pairs of pants in the decomposition of a component may all have cuffs on their boundaries, and the accumulation of these cuffs (as a subset of $\mathbb{C}$) could be a Cantor set. 

A topological end of a surface is a way of going off to infinity on the surface (see \cite{Richards}). More precisely, given a compact exhaustion of a surface, the nested family of complements of the compact sets defines a topological end. This definition is independent of a choice of compact exhaustion (see \cite{Richards}). An end is {\it simple} if not accumulated by other ends or genus. In the case of a Riemann surface with a covering Fuchsian group of the first kind, each simple end is a puncture. 

At this point, we assume that planar, bounded geometry surface $X$ has an infinite part that is homeomorphic to a plane minus countably many closed disks. The boundary of the infinite part is a countable union of closed geodesics, and a pair of pants with two punctures is attached to each geodesic. The planar part is partitioned into right-angled hexagons of bounded side lengths whose three non-adjacent sides lie on the closed geodesic on the boundary of the infinite part, and the other three sides are geodesic arcs orthogonal to the boundary. Necessarily, each closed boundary geodesic of the infinite part meets a bounded number of hexagons.

This situation arises in Kinjo \cite{Kinjo} and Basmajian-Parlier-Vlamis \cite{BPV}. We draw geodesics connecting punctures inside each pair of pants. Further, we connect punctures in one pair of pants to the punctures in another pair of pants if a hexagon inside the infinite part connects the two pairs of pants. The geodesics are drawn such that they are disjoint, and they form an ideal triangulation of $X$.  More generally, the same method gives

\begin{prop}
Let $X$ be a Riemann surface with a planar subsurface $P$ with bounded geometry such that each boundary component of $P$ is attached to a surface with bounded ideal triangulation. Then, $X$  has a bounded ideal triangulation.
\end{prop}

\begin{center}

\includegraphics[width=\linewidth]{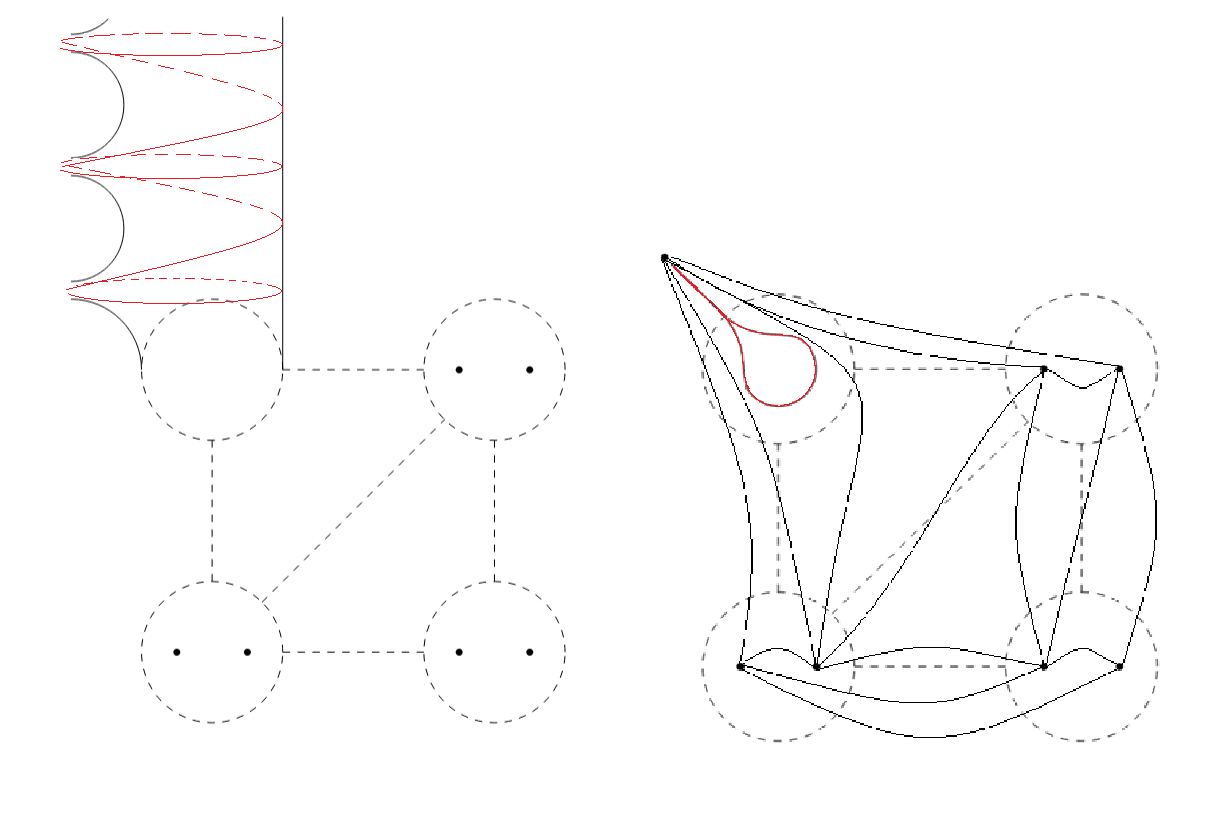}
\end{center}

\begin{center}
    
\begin{tikzpicture}
\draw(0,0) node{Fig 2.1 Extending an ideal  triangulation to the planar part};
\end{tikzpicture}
\end{center}

\section{The shear coordinates on the Farey triangulation and quasisymmetric maps}

Let $\mathbb{H}^2$ be the upper half-plane $\{ z:Im(z)>0\}$ equipped with the hyperbolic metric $\frac{|dz|}{Im(z)}$; the ideal boundary of $\mathbb{H}^2$ is identified with $\mathbb{R}\cup\{\infty\}$. 
Let $\Delta_0$ be the ideal hyperbolic triangle in $\mathbb{H}^2$ with vertices $0$, $1$ and $\infty$. The Modular group $\Gamma$ is generated by the Mobius transformations $z\mapsto \frac{-1}{z}$ and $z\mapsto z+1$. The $\Gamma$-orbit of $\Delta_0$ is a triangulation $\mathcal{F}$ of $\mathbb{H}^2$ called the {\it Farey triangulation}. Every rational number appears as a vertex for the image of $\Delta_0$ under some $\gamma\in\Gamma$. 

The Farey triangulation $\mathcal{F}$ is the image of $\Delta_0$ under all finite compositions of the (orientation reversing) isometric reflections in the sides of $\Delta_0$. In particular, this method of construction reveals that two adjacent triangles - two triangles sharing a common side - are images of each other through a reflection over their common side.

Given two arbitrary, adjacent ideal triangles, we can quantify the degree to which they are ``off-center'' with respect to each other. Adjacent triangles are evenly aligned with each other-``centered''- if the orthogeodesics to their common side from the vertices opposite the common side meet at a point (and form a single geodesic arc): when this is the case we define the {\it shear} between the two triangles to be zero. 

When two orthogeodesics do not meet on the common side-``off-centered'', the shear is the signed distance between their feet on the common side.  To visualize the shear, 
start with two adjacent triangles with zero shear, slide one orthogeodesic of one triangle by the signed  hyperbolic distance equal to the shear along the common side while keeping it orthogonal (see Figure 3.1). The shear is positive if when drawn as a path, this distance is drawn to the left; and negative if it is to the right. This is independent of which triangle we start drawing from. A formula for the shear between two adjacent ideal triangles can be expressed in terms of their vertices. Note that due to the construction of the Farey tesselation, any two adjacent triangles in $\mathcal{F}$ are reflected images of each other over the common shared edge; therefore, all  shears on $\mathcal{F}$ are  zero.

\begin{center}
    \begin{tikzpicture}
        \draw (-5,0)--(5,0);
        \draw (-3,0) arc (180:0:3);
        \draw (1.5,0) arc (180:0:0.75);
        \draw (-1,0) arc (180:0:2);
        \draw (-3,0) arc (180:0:1);
        \draw (-1,0) arc (180:0:1.25);
        \draw [dashed] (-3,0) arc (180:55:1.5); 
        \draw [dashed] (1.5,0) arc (180:145:3);
    \draw [blue,thick,domain=142:60] plot ({cos(\x)+cos(\x)+1}, {sin(\x)+sin(\x});
    \draw (-3,-.5) node{a};
    \draw (-1,-.5) node{b};
    \draw (1.5,-.5) node{c};
    \draw (3,-.5) node{d};
    \draw (0,-1) node{Fig 3.1 Positive shear of a pair of triangles};

    \end{tikzpicture}
    
\end{center}

\begin{center}
    \begin{tikzpicture}
        \draw (-5,0)--(5,0);
        \draw (-3,0) arc (180:0:3);
        \draw (1.5,0) arc (180:0:0.75);
        \draw (-3,0) arc (180:0:2.25);
        \draw (-3,0) arc (180:0:1);
        \draw (-1,0) arc (180:0:1.25);
        \draw [dashed] (3,0) arc (0:110:1.5); 
        \draw [dashed] (-1,0) arc (0:43:3);
    \draw [red,thick,domain=115:40] plot ({2.25*cos(\x)-0.75}, {2.25*sin(\x)});
    \draw (-3,-.5) node{a};
    \draw (-1,-.5) node{b};
    \draw (1.5,-.5) node{c};
    \draw (3,-.5) node{d};
    \draw (0,-1) node{Fig 3.2 Negative shear of a pair of triangles};

    \end{tikzpicture}
    
\end{center}

If two adjacent triangles are known from the context, we will call their shear also the shear of their common side.

To compute the shear of an arbitrary pair of adjacent triangles, let $f\in PSL_2(\mathbb{R})$
be an isometry of $\mathbb{H}^2$ such that $f(a)=-1$, $f(b)=0$ and $f(d)=\infty$. The shear is the hyperbolic distance between the feet of the orthogeodesics and it is invariant under isometries (see Figure 3.2). When the common side has ideal endpoints $0$ and $\infty$, and the other two vertices are $-1$ and $f(c)$, the signed distance between the feet of the orthogeodesics is $\log f(c)$. Since the cross-ratios are invariant under $f\in PSL_2(\mathbb{R})$, it follows that the shear is
$$s=\log\frac{(c-b)(a-d)}{(c-d)(b-a)}.$$ Define the cross ratio $cr(a,b,c,d)=\frac{(c-b)(a-d)}{(c-d)(b-a)}$, then the shear is $\log cr(a,b,c,d)$.

A particularly useful observation is that for a common side with one vertex at $\infty$, if the three  vertices on the real line are equally spaced, the shear on the common side is zero. 

\begin{center}
    \begin{tikzpicture}
        \draw (-2,0)--(3,0);
      \draw (0,0)--(0,3);
      \draw (-1,0)--(-1,3);
      \draw (2,0)--(2,3);
      \draw (-1,0) arc (180:0:.5);
      \draw (0,0) arc (180:0:1);
      \draw [dashed] (-1,0) arc (180:90:1);
      \draw [dashed] (2,0) arc (0:90:2);
      \draw [blue,thick] (0,1)--(0,2);
      \draw (-1,-.5) node{-1};
      \draw (0,-.5) node{0};
      \draw (2,-.5) node{f(c)};
      \draw (0,-1) node{Fig. 3.3 Image of triangles via Mobius transformation $f(z)$};

    \end{tikzpicture}
    
\end{center}

An orientation preserving homeomorphism $h:\mathbb{R}\to\mathbb{R}$ is {\it quasisymmetric map}  if there exists $K\geq 1$ such that for all $x\in\mathbb{R}$, and $t>0$, $$\frac{1}{K}\leq \frac{h(x+t)-h(x)}{h(x)-h(x-t)}\leq K.$$

Let $h:\mathbb{R}\to\mathbb{R}$
be a quasisymmetric map of the real line. For any edge $e$ in the Farey tesselation $\mathcal{F}$, $s_h(e)$ is  the shear of the image $h(e)$ considered as the boundary of the two triangles in the image triangulation $h(\mathcal{F})$. Thus, a quasisymmetric map $h$ induces a shear function
$$
s_h:E(\mathcal{F})\to\mathbb{R},
$$
where $E(\mathcal{F})$ is the set of edges of $\mathcal{F}$.

Given a shear function $s_h$, we can recover the quasisymmetric map $h$ (up to post-composition by an element of $PSL_2(\mathbb{R})$) via the developing map (see \cite{Penner,Saric10}).

In general, an arbitrary  function $s:E(\mathcal{F})\to\mathbb{R}$ induces a developing map $h_s$ from the vertices of the Farey triangulation to $\mathbb{R}$ but $h_s$ does not always extend to a homeomorphism of $\mathbb{R}$-let alone a quasisymmetric map. A necessary and sufficient condition for a shear function to induce a quasisymmetric map is given in \cite{Saric10,Saric21}, which we describe below.

A {\it fan} of edges of $\mathcal{F}$ is the set of all edges of $\mathcal{F}$ that have a common endpoint $p$ on $\mathbb{R}\cup\{\infty\}$. There is a natural ordering on any fan given by the natural orientation (interior points are on the left) of a horocycle whose center is the common endpoint of the fan. We enumerate each fan  by $\{ e_n\}_{n\in\mathbb{Z}}$ such that $e_n$ and $e_{n+1}$ are adjacent edges and $e_n<e_{n+1}$ for the ordering. For any edge $e_m\in\{ e_n\}_{n\in\mathbb{Z}}$ and an integer $k\geq 0$, we define
$$
s(p;m,k)=e^{s(e_m)}\frac{1+e^{s(e_{m+1})}+...e^{s(e_{m+1})+...s(e_{m+k})}}{1+e^{-s(e_{m-1})}+...e^{-s(e_{m-1})-...-s(e_{m-k})}}.
$$
Then (see \cite[Theorem A]{Saric10},\cite[Theorem 1]{Saric21}) a function $s:E(\mathcal{F})\to\mathbb{R}$ induces a quasisymmetric map $h_s:\mathbb{R}\to\mathbb{R}$ if and only if there exists $M\geq 1$ such that for all vertices $p$ of $\mathcal{F}$ and for all $m,k\in\mathbb{Z}$ with $k\geq 0$ we have
\begin{equation}
\label{eqn:qs-cond}
\frac{1}{M}\leq s(p;m,k)\leq M.
\end{equation}

Let $h_n:\mathbb{R}\to\mathbb{R}$ be a sequence of quasisymmetric maps that fix $0$ and $1$, and let $h:\mathbb{R}\to\mathbb{R}$ be a quasisymmetric map that also fixes $0$ and $1$. Then $h_n\to h$ in the Teichm\"uller topology if and only if $h_n\circ h^{-1}$ extend to quasiconformal maps of the upper half-plane with quasiconformal constants $K_n\to 1$ as $n\to\infty$. Let $s_n,s:E(\mathcal{F})\to\mathbb{R}$ be the shear functions corresponding to $h_n,h$. We define
$$
M(s_n,s)=\sup_{p\in\mathcal{F};m,k\in\mathbb{Z},k\geq 0}\Big{\{}\max \Big{(}\frac{s_n(p;m,k)}{s(p;m,k)},\frac{s(p;m,k)}{s_n(p;m,k)}\Big{)}\Big{\}}.$$
Then $h_n\to h$ in the Teichm\"uller topology if and only if $M(s_n,s)\to 1$ as $n\to\infty$ (see \cite[Theorem D]{Saric10}). 

Finally, Parlier and the second author \cite{PS} established a sufficient condition on the shear functions to guarantee the induced maps are quasisymmetric. Namely, if there exists $C>0$ such that for all fans $\{ e\}_{n\in\mathbb{Z}}$ and for all $m,k\in\mathbb{Z}$ with $k\geq 0$ we have
\begin{equation}
\label{eq:qs-suff-cond}
    \Big{|}\sum_{j=m}^{m+k}s(e_j)\Big{|}\leq C
\end{equation}
then the induced map $h_s$ is quasisymmetric. The advantage of the condition (\ref{eq:qs-suff-cond}) compared to (\ref{eqn:qs-cond}) is that it is easier to check. For a condition on shears for $C^{1+\alpha}$-homeomorphisms, see \cite{SWW}.

\section{The shear coordinates for infinite surfaces and some examples}
\label{sec:examples}

Motivated by the ideal triangulations of $\mathbb{H}^2$ and the sufficient condition for quasisymmetry (\ref{eq:qs-suff-cond}), we introduce the notion of a bounded ideal triangulation on $X$.
\begin{definition}
An ideal hyperbolic triangulation $\mathcal{T}$ of $X$ is said to be bounded if both rays of each edge of $\mathcal{T}$ end at punctures of $X$, there is an upper bound on the number of incident edges of $\mathcal{T}$ at punctures of $X$, and the absolute values of the shears on the edges of $\mathcal{T}$ are bounded above. 
\end{definition}

Let $X$ be a  Riemann surface $X$ with a bounded ideal triangulation $\mathcal{T}$. Let $f:X\to X'$ be a quasiconformal map. Then $f(\mathcal{T})$ can be homotoped modulo ideal endpoints into hyperbolic triangulation $\mathcal{T}'$ of $X'$ which has the same number of incident edges at each puncture of $X'$. Note that the sums of the shears at the punctures of $X'$ are necessarily zero (otherwise the end would not be a puncture). We have

\begin{prop}
The Riemann surface $X$ with a bounded ideal hyperbolic triangulation is quasiconformal to the Riemann surface $X_0$ with bounded ideal triangulation whose all shears are zero. Consequently, the covering group of $X_0$ is a subgroup of $PSL_2(\mathbb{Z})$.
\end{prop}
\begin{proof}
Lift $\mathcal{T}$ to an ideal triangulation $\tilde{\mathcal{T}}$ of $\mathbb{H}^2$. The shears on $\tilde{\mathcal{T}}$ are invariant under the covering group of $X$. Let $h:\mathbb{R}\to\mathbb{R}$ be a homeomorphism which is induced by mapping the triangles of the Farey triangulation $\mathcal{F}$ onto the triangles of $\tilde{\mathcal{T}}$. 
Since the shear function of $X$ is in $\ell^{\infty}(E(\mathcal{F}),\mathbb{R})$, by Theorem \ref{thm:Teich-param}, the homeomorphism $h$ is  quasisymmetric. By the invariance of the shear function on $\tilde{\mathcal{T}}$, the quasisymmetric map $h$ extends to a quasiconformal map of $\mathbb{H}^2$ that conjugates a subgroup of $PSL_2(\mathbb{Z})$ onto the covering group of $X$. The quotient of $\mathbb{H}^2$ by this subgroup of $PSL_2(\mathbb{Z})$ is a Riemann surface $X_0$ with an ideal triangulation $\mathcal{T}_0$ with zero shears (the quotient of $\mathcal{F}$).
\end{proof}

We discuss examples of surfaces (with bounded geometry without bounded pants decompositions)  where shear calculations can provide insight into their structure. The first example is a flute surface introduced by Basmajian-Parlier-Vlamis \cite{BPV}  which is homeomorphic to $\mathbb{R}^2-\mathbb{Z}^2$ with \textit{bounded geometry} but where there is no bounded pants decomposition. We show that this surface is parabolic and that its covering group is a subgroup of $PSL_2(\mathbb{Z})$. Next, we consider a surface introduced by Kinjo\cite{Kinjo} that also has a bounded geometry and no bounded pants decomposition but is not parabolic, and we compute shears for this surface. 

\subsection{A Parabolic Flute Surface with Bounded Ideal Triangulation}(see \cite{BPV}) For any real number $b>arcsinh(1)$, there exists a unique right-angled pentagon with two adjacent sides of lengths $b$. As the values of $b$ approach $arcsinh(1)$, the pentagon degenerates into a right-angled ideal quadrilateral called a Lambert quadrilateral. We start to construct a surface $X$ by gluing by isometries four copies of a Lambert quadrilateral, $b=arcsinh(1)$, along their infinite sides (see Figure 4.1.1). The gluings are such that the obtained surface $X_b$ has one boundary, a closed piecewise hyperbolic arc with four geodesic arcs of length $2b$ meeting at right angles. The other boundary component is the single puncture, and $X_b$ is called a tile.  The surface $X$
is formed by isometrically gluing countably many isometric copies of $X_b$ along their piecewise geodesic boundaries such that each vertex belongs to four copies of $X_b$, which makes the angle $2\pi$. The surface $X$ has a subgroup of hyperbolic isometries isomorphic to $\mathbb{Z}^2$ that maps a single copy of $X_b$ to any other copy.
The subgroup of isometries also maps one puncture to any other puncture and $X$ is homeomorphic to $\mathbb{R}^2-\mathbb{Z}^2$. In fact, both $X$ and  $\mathbb{R}^2-\mathbb{Z}^2$ are conformal coverings of the square torus and therefore they are conformal themselves.

\begin{center}

\begin{tikzpicture}
\draw (-2.2,1) node{b};
\draw (-1,2.2) node{b};

\draw (0,-2.5) node{Fig. 4.1.1 $X_b$ tile of $X$};

\draw (0,0) node[circle,fill,inner sep=1pt,label=right:{}]{};
\draw (-2,2)--(2,2);
\draw (-2,2)--(-2,-2);
\draw (2,2)--(2,-2);
\draw (-2,-2)--(2,-2);
\draw (2,0)--(-2,0);
\draw (0,2)--(0,-2);
\draw (-1.8,2)--(-1.8,1.8);
\draw (-1.8,1.8)--(-2,1.8);
\draw (-1.8,0)--(-1.8,0.2);
\draw(-1.8,0.2)--(-2,0.2);
\draw(-0.2,2)--(-0.2,1.8);
\draw(-0.2,1.8)--(0,1.8);

\end{tikzpicture}

\end{center}

\begin{prop}
    There exists a triangulation $\mathcal{T}$ of the Riemann surface $X$  described abovethat is bounded, in particular such that the shears on all edges are zero. Thus, the covering group is a subgroup of $PSL(2,\mathbb{Z})$.
\end{prop}
\begin{proof}
We use the homeomorphism of $X$ and $\mathbb{R}^2-\mathbb{Z}^2$ to identify the subgroup of the isometry group which permutes the tiles $X_b$ of $X$ with the action of $\mathbb{Z}^2$ on $\mathbb{R}^2-\mathbb{Z}^2$.  Fix one puncture $v_1$ to be identified with the origin $(0,0)\in\mathbb{R}^2$. Choose 3 punctures $v_2,v_3,v_4$ such that $v_2,v_3,v_4$ are images of $(1,0),(0,1),(1,1)\in\mathbb{Z}^2$ of $v_1$ under the isometric action of $\mathbb{Z}^2$ that forms the surface. To define an ideal triangulation of $X$, form an edge (ideal hyperbolic geodesic) between punctures if they contain sides of a Lambert quadrilateral along with an edge between $v_1$ and $(1,1)v_1$. The triangulation of the rest of the surface is as in Figure 4.1.2 where the added edges are in a zig-zag fashion in the horizontal direction.  This forms a graph with vertices with degrees $8$ and $4$ as in Figure 4.1.2 below.

\begin{center}

\begin{tikzpicture}

\draw (-.5,1.2) node{b};

\draw (-1.2,0.5) node{b};

\draw (1,-3) node{Fig. 4.1.2 Triangulation of $X_b$ with zero shears.};

\draw (0,0) node[circle,fill,inner sep=1pt,label=right:{}]{};
\draw (4,2) node[circle,fill,inner sep=1pt,label=right:{}]{};
\draw (4,0) node[circle,fill,inner sep=1pt,label=right:{}]{};
\draw (4,-2) node[circle,fill,inner sep=1pt,label=right:{}]{};
\draw (2,2) node[circle,fill,inner sep=1pt,label=right:{}]{};
\draw (2,0) node[circle,fill,inner sep=1pt,label=right:{}]{};
\draw (2,-2) node[circle,fill,inner sep=1pt,label=right:{}]{};
\draw (0,2) node[circle,fill,inner sep=1pt,label=right:{}]{};
\draw (0,-2) node[circle,fill,inner sep=1pt,label=right:{}]{};
\draw (-2,2) node[circle,fill,inner sep=1pt,label=right:{}]{};
\draw (-2,0) node[circle,fill,inner sep=1pt,label=right:{}]{};
\draw (-2,-2) node[circle,fill,inner sep=1pt,label=right:{}]{};

\draw (-2,2)--(2,2);
\draw (-2,2)--(-2,-2);
\draw (2,2)--(2,-2);
\draw (-2,-2)--(2,-2);
\draw (2,0)--(-2,0);
\draw (0,2)--(0,-2);
\draw (2,2)--(4,2);
\draw (4,2)--(4,-2);
\draw(2,0)--(4,0);
\draw (4,-2)--(0,-2);
\draw (0,0)--(-2,2);
\draw (0,0)--(-2,-2);
\draw(0,0)--(2,2);
\draw (0,0)--(2,-2);
\draw (2,2)--(4,0);
\draw (4,0)--(2,-2);
\draw (0,2)--(0,2.2);
\draw (-2,2)--(-1.8,2.2);
\draw (-2,2)--(-2,2.2);
\draw (-2,2)--(-2.2,2);
\draw (-2,2)--(-2.2,2.2);
\draw (-2,2)--(-2.2,1.8);
\draw (-2,0)--(-2.2,0);
\draw (-2,-2)--(-2.2,-2);
\draw (-2,-2)--(-2,-2.2);
\draw (-2,-2)--(-2.2,-2.2);
\draw (-2,-2)--(-1.8,-2.2);
\draw (-2,-2)--(-2.2,-1.8);
\draw (0,-2)--(0,-2.2);
\draw (2,2)--(2,2.2);
\draw (2,2)--(1.8,2.2);
\draw (2,2)--(2.2,2.2);
\draw(4,2)--(4,2.2);
\draw (4,2)--(4.2,2);
\draw (4,0)--(4.2,0);
\draw (4,0)--(4.2,0.2);
\draw (4,0)--(4.2,-0.2);
\draw (4,-2)--(4.2,-2);
\draw (4,-2)--(4,-2.2);
\draw (2,-2)--(2,-2.2);
\draw (2,-2)--(1.8,-2.2);
\draw (2,-2)--(2.2,-2.2);

\draw [dashed] (1,2)--(1,-2);
\draw [dashed] (3,2)--(3,-2);
\draw [dashed] (4,1)--(-2,1);
\draw [dashed] (4,-1)--(-2,-1);
\draw [dashed] (-1,2)--(-1,-2);

\end{tikzpicture}

\end{center}
Since any two adjacent triangles are reflections of each other over their shared edge, the shear of every edge must be zero. Therefore, this is a bounded triangulation such that the shear of every edge is zero and by above, the covering group is a subgroup of $PSL(2,\mathbb{Z})$. 
\end{proof}

A Green's function on $X$ is a harmonic function $u:X\setminus\{ z_0\}\to\mathbb{R}$ such that $u(z)\asymp \log \frac{1}{|z-z_0|}$ as $z\to z_0$, and $u(z)\to 0$ as $z\to\partial_{\infty}X$.  A Riemann surface without Green's function is called {\it parabolic}. There are numerous equivalent conditions for a Riemann surface to be parabolic. For example, a Riemann surface $X$ is parabolic if and only if the Brownian motion on $X$ is ergodic (see Ahlfors-Sario \cite{AhlforsSario}) if and only if the geodesic flow on the unit tangent bundle $T^1X$ of $X$ is ergodic (Hopf-Tsuji-Sullivan theorem, see \cite{Sullivan}) if and only if the Poincar\'e series of the Fuchsian covering group $\Gamma$ of $X$ is divergent.

We prove that $X$ is a parabolic surface.

\begin{prop}
    The Riemann surface $X$ as described above is parabolic.
\end{prop}
\begin{proof}

Any puncture $x$ of the surface $X$ is on the ideal boundary of a tile isometric to $X_b$. Each geodesic arc (of length $2b$) of the piecewise geodesic boundary of $X_b$ is between two hypercycles on the distance $\frac{1}{4}b$. We form a rectangle by extending by $\frac{1}{4}b$ both hypercycles in both direction and capping off by geodesic arcs orthogonal to them. 

\begin{center}

\begin{tikzpicture}
\label{fig:par}
\draw (0,-4.75) node{Fig. 4.1.3 $A_1$ and $A_2$};

\draw (0,0) node[circle,fill,inner sep=1pt,label=right:{}]{};
\draw (4,2) node[circle,fill,inner sep=1pt,label=right:{}]{};
\draw (4,0) node[circle,fill,inner sep=1pt,label=right:{}]{};
\draw (4,-2) node[circle,fill,inner sep=1pt,label=right:{}]{};
\draw (2,2) node[circle,fill,inner sep=1pt,label=right:{}]{};
\draw (2,0) node[circle,fill,inner sep=1pt,label=right:{}]{};
\draw (2,-2) node[circle,fill,inner sep=1pt,label=right:{}]{};
\draw (0,2) node[circle,fill,inner sep=1pt,label=right:{}]{};
\draw (0,-2) node[circle,fill,inner sep=1pt,label=right:{}]{};
\draw (-2,2) node[circle,fill,inner sep=1pt,label=right:{}]{};
\draw (-2,0) node[circle,fill,inner sep=1pt,label=right:{}]{};
\draw (-2,-2) node[circle,fill,inner sep=1pt,label=right:{}]{};
\draw (-4,0) node[circle,fill,inner sep=1pt,label=right:{}]{};
\draw (-4,2) node[circle,fill,inner sep=1pt,label=right:{}]{};
\draw (-4,4) node[circle,fill,inner sep=1pt,label=right:{}]{};
\draw (-4,-2) node[circle,fill,inner sep=1pt,label=right:{}]{};
\draw (-4,-4) node[circle,fill,inner sep=1pt,label=right:{}]{};
\draw (-2,-4) node[circle,fill,inner sep=1pt,label=right:{}]{};
\draw (-2,4) node[circle,fill,inner sep=1pt,label=right:{}]{};
\draw (0,-4) node[circle,fill,inner sep=1pt,label=right:{}]{};
\draw (0,4) node[circle,fill,inner sep=1pt,label=right:{}]{};
\draw (2,4) node[circle,fill,inner sep=1pt,label=right:{}]{};
\draw (2,-4) node[circle,fill,inner sep=1pt,label=right:{}]{};
\draw (4,4) node[circle,fill,inner sep=1pt,label=right:{}]{};
\draw (4,-4) node[circle,fill,inner sep=1pt,label=right:{}]{};

\draw (-4.5,0)--(4.5,0);
\draw (-4.5,1)--(4.5,1);
\draw (-4.5,2)--(4.5,2);
\draw (-4.5,3)--(4.5,3);
\draw (-4.5,4)--(4.5,4);

\draw (-4.5,-1)--(4.5,-1);
\draw (-4.5,-2)--(4.5,-2);
\draw (-4.5,-3)--(4.5,-3);
\draw (-4.5,-4)--(4.5,-4);

\draw (-4,4.5)--(-4,-4.5);
\draw (-3,4.5)--(-3,-4.5);
\draw (-2,4.5)--(-2,-4.5);
\draw (-1,4.5)--(-1,-4.5);
\draw (0,4.5)--(0,-4.5);
\draw (1,4.5)--(1,-4.5);
\draw (2,4.5)--(2,-4.5);
\draw (3,4.5)--(3,-4.5);
\draw (4,4.5)--(4,-4.5);

\draw [dashed] (-4.5,.75)--(4.5,.75);
\draw [dashed] (-4.5,1.25)--(4.5,1.25);
\draw [dashed] (-4.5,2.75)--(4.5,2.75);
\draw [dashed] (-4.5,3.25)--(4.5,3.25);

\draw [dashed] (-4.5,-.75)--(4.5,-.75);
\draw [dashed] (-4.5,-1.25)--(4.5,-1.25);
\draw [dashed] (-4.5,-2.75)--(4.5,-2.75);
\draw [dashed] (-4.5,-3.25)--(4.5,-3.25);

\draw [dashed] (.75,-4.5)--(.75,4.5);
\draw [dashed] (1.25,-4.5)--(1.25,4.5);
\draw [dashed] (2.75,-4.5)--(2.75,4.5);
\draw [dashed] (3.25,-4.5)--(3.25,4.5);

\draw [dashed] (-.75,-4.5)--(-.75,4.5);
\draw [dashed] (-1.25,-4.5)--(-1.25,4.5);
\draw [dashed] (-2.75,-4.5)--(-2.75,4.5);
\draw [dashed] (-3.25,-4.5)--(-3.25,4.5);

\fill[black,opacity=0.3] (-1.25,-1.25) -- (-.75,-1.25) -- (-.75,1.25) -- (-1.25,1.25);

\fill[black,opacity=0.3] (-.75,-.75) -- (-.75,-1.25) -- (1.25,-1.25)--(1.25,-.75);

\fill[black,opacity=0.3] (.75,1.25) -- (1.25,1.25) -- (1.25,-.75)--(.75,-.75);

\fill[black,opacity=0.3] (.75,1.25) -- (.75,.75) -- (-.75,.75)--(-.75,1.25);

\fill[black,opacity=0.3] (-3.25,-3.25) -- (-2.75,-3.25) -- (-2.75,3.25)--(-3.25,3.25);

\fill[black,opacity=0.3] (-2.75,-2.75) -- (-2.75,-3.25) -- (3.25,-3.25)--(3.25,-2.75);

\fill[black,opacity=0.3] (2.75,-2.75) -- (3.25,-2.75) -- (3.25,3.25)--(2.75,3.25);

\fill[black,opacity=0.3] (2.75,3.25) -- (2.75,2.75) -- (-2.75,2.75)--(-2.75,3.25);

\end{tikzpicture}

\end{center}

We form a sequence $C_n$ of simple closed piecewise geodesic polygonal arcs by taking the geodesic arcs of the boundaries of the tiles. First, take $C_1$ to be exactly the boundary of a single tile. Then, $C_1$ separates a single puncture from the rest of the punctures. In the next step, $C_2$ is the piecewise geodesic boundary of the union of the tiles attached to $C_1$ (including attached tiles to the vertices of $C_1$). If we are given $C_n$, the next piecewise closed arc $C_{n+1}$ is the boundary of the tiles attached to $C_n$. Let $A_n$ be a topological annulus formed by the union of the rectangles (defined above) containing the geodesic arcs of $C_n$ considered as boundaries of the tiles (see Figure 4.1.3). Then $A_n$ is disjoint from $A_{n+1}$ and any curve leaving $A_1$ and going to the non-simple end of $X$, i.e., not the puncture,  must cross each $A_n$ for $n>1$.

The hyperbolic length of each $C_n$ is $(4n-2)b$ and it encloses $(2n-1)^{2}$ punctures. 
Let $\Gamma_n$ be the family of all curves in the annular region $A_n$ connecting its boundary components. The {\it extremal length} of $\Gamma_n$ is defined by the following (see \cite{Ahlfors, GarnettMarshall}):

$$ext(\Gamma_n)=\sup_{\rho}\frac{L_{\rho}(\Gamma_n)^2}{A(\rho)},$$
where the supremum is over all Borel measurable functions $\rho \geq 0$ in $A_n$, $L_{\rho}(\Gamma_n)=\inf_{\gamma\in \Gamma_n}\int_\gamma \rho (z)|dz|$,  and $A(\rho)=\int_{A_n}\rho ^2(z)dxdy$ with $0<A(\rho )<\infty$.

To find a lower bound on $ext(\Gamma_n)$, we choose a metric $\rho$ to be the hyperbolic metric on $X$ and compute $L_{\rho}(\Gamma_n)/A(\rho )$. The hyperbolic area $A(\rho )$ of $A_n$ is at most $32bn$ and the minimal length $L_{\rho}(\gamma )$ is the minimal hyperbolic distance between two boundary components of $A_n$, which $\frac{1}{2}b$. Therefore
$$
ext(\Gamma_n)\geq \frac{C}{n}$$
for some constant $C>0$ independent of $n$. 

Let $\Gamma$ be the family of all curves connecting $A_1$ to the non-simple end of $X$. Then each curve in $\Gamma$ overflows a curve in $\Gamma_n$ for each $n$. By the serial rule (see \cite[page 135]{GarnettMarshall}) we have
$$
ext(\Gamma )\geq\sum_{n=2}^{\infty}ext(\Gamma_n)\geq\sum_{n=2}^{\infty}\frac{C}{n}=\infty .$$ By \cite{AhlforsSario}, the surface $X$ is parabolic.
\end{proof}

\subsection{Non-parabolic Flute Surface with Bounded Ideal Triangulation}

Let $T$ be a hyperbolic triangle with angles $\alpha=\frac{\pi}{2}$, $\beta=\frac{\pi}{4}$, $\gamma=\frac{\pi}{8}$  in a clockwise direction. Let  $a, b, c$  be the sides of $T$ opposite the angles $\alpha$, $\beta$, $\gamma$. Since $\frac{1}{\alpha}+\frac{1}{\beta}+\frac{1}{\gamma}<1$, it is known that successive reflections about the sides of $T$ tesselate the hyperbolic plane (for example, see Caratheodory \cite[pg. 177-182]{Caratheodory}). 
For the simplicity of notation, denote by $a,b,c$ the hyperbolic reflections in the sides $a,b,c$, respectively. Let $\Delta(2,4,8)$ be the \textit{Triangle Group} generated by reflections  $a,b,c$ with the represention $\langle a,b,c:a^2=b^2=c^2=(bc)^2=(ac)^4=(ab)^8=1\rangle$. The triangle $T$ is a fundamental domain of $\Delta(2,4,8)$. 

The triangle group $\Delta(2,4,8)$ has an index 2 subgroup consisting of orientation-preserving elements. 
The element $B=ca$  is a clockwise rotation around $z_b$ by the angle $2\beta=\frac{\pi}{2}$.   Similarly $A=bc$ is a clockwise rotation about $z_{a}$ by the angle $2\alpha=\pi$ and $C=ab$ is a clockwise rotation about $z_c$ by the angle $2\gamma =\frac{\pi}{4}$. As $a,b,c$  have order $2$ in $\Delta(2,4,8)$, we have a relation $ABC=1$ implying $CA=B^{-1}$.  
 
This orientation preserving subgroup $G$  of  $\Delta(2,4,8)$ is called the \textit{Von-Dyke Group} and has a representation $G=\langle A,C : A^2=C^8=(AC)^4=1\rangle$. The triangle $\Delta=T\cup bT$ with vertices $z_b,z_c$ and $z_{\sigma}=bz_b$ is a fundamental domain for $G$.

Now, we recall the construction of a non-parabolic surface from Kinjo \cite[section 2]{Kinjo}. Let $\epsilon>0$ be smaller than half of the minimum distances of $z_b,z_c,z_{\sigma}$ from opposite sides of $\Delta$. Place the triangle $T$ in $\mathbb{H}^2$ such that $z_a=i$ and $z_c$ is on the positive imaginary axis above $i$. 
Let $c_{\epsilon}$ be the point on the segment $[z_a,z_c]$ with hyperbolic distance $\epsilon$ away from $z_c$. Let $b_{\epsilon}$ be the point on the distance $\epsilon$ from $z_b$ along the orthogonal geodesic arc from $z_b$ to the opposite side in $\Delta$. Let $\sigma_{\epsilon}$ be the point on the distance $\epsilon$ from $z_{\sigma}$ along orthogonal  geodesic arc from $z_{\sigma}$ to the opposite side in $\Delta$. 
The surface $X_{\Delta}=\mathbb{H}^2-\{gz_b,gz_c,gz_{\sigma},gb_{\epsilon},gc_{\epsilon},g\sigma_{\epsilon}:\forall g\in G\}$ is parabolic since $h(z)=\log |\frac{z-i}{z+i}|$ is a Green's function on $X_{\Delta}$ (see \cite[ pg. 203-206]{AhlforsSario}).

Now we construct a triangulation $\mathcal{T}_{\Delta}=(V_{\Delta},E_{\Delta})$ of $X_{\Delta}$. The vertex set, $V_{\Delta}$, are the punctures of $X_{\Delta}$, partitioned into $\Delta$-vertices and $\epsilon$-vertices: vertices of $g\Delta$ and vertices $\epsilon$ distance from $\Delta$-vertices in $\mathbb{H}^2$ respectively.  For the edge set, $E_{\Delta}$, we include the following edges. Within every $g(\Delta )$, connect $\epsilon$-vertices together by an edge. Every $\Delta$-vertex connect by an edge to all $\epsilon$-vertices $\epsilon$ distance away in $\mathbb{H}^2$; moreover, connect these $\epsilon$-vertices in sequence cyclically. The complement of the edges introduced so far consists of triangles that have one vertex in $\{ g(z_b),g(z_c),g(z_{\sigma}):g\in G\}$ and triangles with three vertices $\{ (g(b_{\epsilon}),g(c_{\epsilon}),g(\sigma_{\epsilon})):g\in G\}$, and rectangles that intersect at edges of $g(\Delta )$, for some $g\in G$. 
Lastly, include edges $[gb_{\epsilon},gCc_{\epsilon}]$ and $[g\sigma_{\epsilon},gA\sigma_{\epsilon}]$ for $g\in G$, 
the diagonals of the remaining rectangles. See figure $4.2.1$ below for the triangles incident to $\Delta$.

\begin{center}
    \begin{tikzpicture}

\draw (0,-1.25) node{Fig. 4.2.1 Triangles in $\mathcal{T}_{\Delta}$ incident to $\Delta$ };

\tkzDefPoint(-4,0){A}\tkzDefPoint(-1.5,1.5){B}\tkzDefPoint(0,4){C}
\tkzCircumCenter(A,B,C)\tkzGetPoint{O}
\tkzDrawArc [dashed](O,A)(C);

\tkzDefPoint(0,4){A}\tkzDefPoint(1.5,1.5){B}\tkzDefPoint(4,0){C}
\tkzCircumCenter(A,B,C)\tkzGetPoint{O}
\tkzDrawArc [dashed](O,A)(C);

\draw [dashed](-4,0)--(4,0);

\tkzDefPoint(-2.25,.35){A}\tkzDefPoint(-1,1){B}\tkzDefPoint(0,2){C}
\tkzCircumCenter(A,B,C)\tkzGetPoint{O}
\tkzDrawArc (O,A)(C);

\tkzDefPoint(0,2){A}\tkzDefPoint(1,1){B}\tkzDefPoint(2.25,.35){C}
\tkzCircumCenter(A,B,C)\tkzGetPoint{O}
\tkzDrawArc (O,A)(C);

\tkzDefPoint(-2.25,.35){A}\tkzDefPoint(0,.25){B}\tkzDefPoint(2.25,.35){C}
\tkzCircumCenter(A,B,C)\tkzGetPoint{O}
\tkzDrawArc (O,A)(C);

\tkzDefPoint(2.25,-.35){A}\tkzDefPoint(0,-.25){B}\tkzDefPoint(-2.25,-.35){C}
\tkzCircumCenter(A,B,C)\tkzGetPoint{O}
\tkzDrawArc (O,A)(C);

\draw (.35,4) node{\begin{small}$z_c$\end{small}};
\draw (-4.1,.25) node{\begin{small}$z_b$\end{small}};

\draw (4.1,.25) node{\begin{small}$z_{\sigma}$\end{small}};

\draw (-1.5,3) node{\begin{small}$C c_{\epsilon}$\end{small}};

\draw (0,2) node[circle,fill,inner sep=1pt,label=right:{}]{};
\draw (0.4,2) node{\begin{small}$c_{\epsilon}$\end{small}};

\draw (2.25,.35) node[circle,fill,inner sep=1pt,label=right:{}]{};
\draw (-1.9,.2) node{\begin{small}$b_{\epsilon}$\end{small}};

\draw (-1.9,-.55) node{\begin{small}$A\sigma_{\epsilon}$\end{small}};

\draw (2.2,.55) node{\begin{small}$\sigma_{\epsilon}$\end{small}};

\draw (2.25,-.35) node[circle,fill,inner sep=1pt,label=right:{}]{};

\draw (-2.25,.35) node[circle,fill,inner sep=1pt,label=right:{}]{};

\draw (-2.25,-.35) node[circle,fill,inner sep=1pt,label=right:{}]{};

\draw (0,4) node[circle,fill,inner sep=1pt,label=right:{}]{};

\draw (4,0) node[circle,fill,inner sep=1pt,label=right:{}]{};

\draw (-4,0) node[circle,fill,inner sep=1pt,label=right:{}]{};

\draw (0,4)--(0,2);

\draw (-4,0)--(-2.25,.35);
\draw (4,0)--(2.25,.35);
\draw (-4,0)--(-2.25,-.35);
\draw (4,0)--(2.25,-.35);

\tkzDefPoint(-2.25,-.35){A}\tkzDefPoint(-2.2,0){B}\tkzDefPoint(-2.25,.35){C}
\tkzCircumCenter(A,B,C)\tkzGetPoint{O}
\tkzDrawArc (O,A)(C);

\tkzDefPoint(2.25,.35){A}\tkzDefPoint(2.2,0){B}\tkzDefPoint(2.25,-.35){C}
\tkzCircumCenter(A,B,C)\tkzGetPoint{O}
\tkzDrawArc (O,A)(C);

\draw (2.25,.35)--(-2.25,-.35);

\tkzDefPoint(-2.25,-.35){A}\tkzDefPoint(-2.2,0){B}\tkzDefPoint(-2.25,.35){C}
\tkzCircumCenter(A,B,C)\tkzGetPoint{O}
\tkzDrawArc (O,A)(C);

\draw (0,4)--(0,2);
\draw (-4,0)--(-2.25,.35);
\draw (4,0)--(2.25,.35);
\draw (-4,0)--(-2.25,-.35);

\draw (-1,3) node[circle,fill,inner sep=1pt,label=right:{}]{};
\draw (1,3) node[circle,fill,inner sep=1pt,label=right:{}]{};

\tkzDefPoint(-1,3){A}\tkzDefPoint(-.25,3.5){B}\tkzDefPoint(0,4){C}
\tkzCircumCenter(A,B,C)\tkzGetPoint{O}
\tkzDrawArc (O,A)(C);

\tkzDefPoint(0,4){A}\tkzDefPoint(.25,3.5){B}\tkzDefPoint(1,3){C}
\tkzCircumCenter(A,B,C)\tkzGetPoint{O}
\tkzDrawArc (O,A)(C);

\tkzDefPoint(-1,3){A}\tkzDefPoint(-.25,2.15){B}\tkzDefPoint(0,2){C}
\tkzCircumCenter(A,B,C)\tkzGetPoint{O}
\tkzDrawArc (O,A)(C);

\tkzDefPoint(0,2){A}\tkzDefPoint(.25,2.15){B}\tkzDefPoint(1,3){C}
\tkzCircumCenter(A,B,C)\tkzGetPoint{O}
\tkzDrawArc (O,A)(C);

\draw (-3,.75) node[circle,fill,inner sep=1pt,label=right:{}]{};
\draw (3,.75) node[circle,fill,inner sep=1pt,label=right:{}]{};

\tkzDefPoint(-2.25,.35){A}\tkzDefPoint(-2.5,.6){B}\tkzDefPoint(-3,.75){C}
\tkzCircumCenter(A,B,C)\tkzGetPoint{O}
\tkzDrawArc (O,A)(C);

\tkzDefPoint(3,.75){A}\tkzDefPoint(2.5,.6){B}\tkzDefPoint(2.25,.35){C}
\tkzCircumCenter(A,B,C)\tkzGetPoint{O}
\tkzDrawArc (O,A)(C);

\tkzDefPoint(-4,0){A}\tkzDefPoint(-3.5,.25){B}\tkzDefPoint(-3,.75){C}
\tkzCircumCenter(A,B,C)\tkzGetPoint{O}
\tkzDrawArc (O,A)(C);

\tkzDefPoint(3,.75){A}\tkzDefPoint(3.5,.25){B}\tkzDefPoint(4,0){C}
\tkzCircumCenter(A,B,C)\tkzGetPoint{O}
\tkzDrawArc (O,A)(C);

\tkzDefPoint(-3,.75){A}\tkzDefPoint(-2,1.5){B}\tkzDefPoint(-1,3){C}
\tkzCircumCenter(A,B,C)\tkzGetPoint{O}
\tkzDrawArc (O,A)(C);

\tkzDefPoint(1,3){A}\tkzDefPoint(2,1.5){B}\tkzDefPoint(3,.75){C}
\tkzCircumCenter(A,B,C)\tkzGetPoint{O}
\tkzDrawArc (O,A)(C);

\tkzDefPoint(-2.25,.35){A}\tkzDefPoint(-1.85,.75){B}\tkzDefPoint(-1,3){C}
\tkzCircumCenter(A,B,C)\tkzGetPoint{O}
\tkzDrawArc (O,A)(C);

\tkzDefPoint(0,2){A}\tkzDefPoint(2,1){B}\tkzDefPoint(3,.75){C}
\tkzCircumCenter(A,B,C)\tkzGetPoint{O}
\tkzDrawArc (O,A)(C);

    \end{tikzpicture}
\end{center}

\begin{prop}
$\mathcal{T}_{\Delta}$ as defined above is a triangulation for $X_{\Delta}$ that is invariant under the action of  the conformal group of symmetries $G$ of $X_{\Delta}$.  In particular, $\mathcal{T}_{\Delta}$ is a bounded ideal triangulation of $X_{\Delta}$.   
\end{prop}

\begin{proof}

The edges listed above other than the diagonals of rectangles are
invariant under the action of $G$
since the vertices $V_{\Delta}$ are given in terms of the orbits of the vertices of $\Delta$ (which is a fundamental domain for $G$) and the points $\epsilon$ distance away from this orbit  which is invariant under isometries of $\mathbb{H}^2$. The ideal triangles inside the orbit $\{ g\Delta :g\in G\}$ are invariant under $G$ and so are the ideal triangles  with one vertex at a point of the orbit of vertices of $\Delta$. 

Therefore, the ideal rectangles are also invariant under the action of $G$ since they are complementary to the ideal triangles.  The diagonals in the ideal rectangles are defined by choosing  diagonals in two rectangles and propagating by the  action of $G$. The first rectangle $R_1$ intersects the edge $[z_b,z_c]$ and the diagonal is $[b_{\epsilon},Cc_{\epsilon}]$.  The choice of the diagonals in the orbit of this rectangle is given by $g[b_{\epsilon},Cc_{\epsilon}]$. Since we are constructing a triangulation, we require that each rectangle in this process has only one diagonal. To see this, it is enough to prove that any rectangle (in the orbit) is the image of the first rectangle under exactly one element of $G$.  Let $gR_1=g'R_1$ for $g,g'\in G$. Then $(g')^{-1}gR_1=R_1$. If $(g')^{-1}g\neq 1$, then $(g')^{-1}g(z_b)=z_c$. Then $(g')^{-1}g$ conjugates the subgroup  $Fix(z_b)$ of $G$ that fixes $z_b$ to the subgroup $Fix(z_c)$ that fixes $z_c$. Since $Fix(z_b)$ has order $4$ and $Fix(z_c)$ has order $8$, this is impossible. Thus $(g')^{-1}g=id$ and the uniqueness of the diagonals follows. 

The second rectangle $R_2$ intersects the side $[z_b,z_{\sigma}]$ of $\Delta$. If $gR_2=g'R_2$ for $g,g'\in G$, then $(g')^{-1}gR_2=R_2$. This implies that either $(g')^{-1}g=1$ or $(g')^{-1}g=A$. 
Even though the choice of $g$ is not unique, the diagonal in each rectangle in the orbit is unique since $A$ fixes the diagonal $[\sigma_{\epsilon}, A\sigma_{\epsilon}]$ . Therefore $\mathcal{T}_{\Delta}$ is a triangulation that is invariant under the action of $G$.

By the construction, each vertex of $\mathcal{T}_{\Delta}$ has valence at most $8$ and there are finitely many orbits of rectangles in $\mathcal{T}_{\Delta}$. Therefore the shears are bounded and the triangulation is bounded. 

It remains to prove that the whole surface $X_{\Delta}$ is covered by the triangulation $\mathcal{T}_{\Delta}$. Notice that the lift of the triangulation $\mathcal{T}_{\Delta}$ to the universal covering $\mathbb{H}^2$ is a triangulation whose shears satisfy (\ref{eq:qs-suff-cond}) and therefore the triangulation covers $\mathbb{H}^2$. 
\end{proof}

\section{The manifold structure on the Teichm\"uller and quasi-Fuchsian spaces}

The shear coordinates are real-valued and naturally belong to a real vector space. Therefore, even though the Teichm\"uller space is a complex Banach manifold, the shear map can only be a real-analytic diffeomorphism. 
We recall a natural manifold structure on the Teichm\"uller and Quasi-Fuchsian space, which (we will prove in the subsequent sections) makes the shear map real-analytic on the Teichm\"uller space and extendable to a holomorphic map in a small neighborhood of the Teichm\"uller space. In this setting, the Teichm\"uller space embeds as a totally real submanifold in the (complex Banach manifold) quasi-Fuchsian space.

We conformally identify the Riemann surface $X$ with $\HH /\Gamma$, where $\Gamma$ is a Fuchsian group acting on the upper half-plane $\HH$. The Teichm\"uller space $T(X)$ consists of all Teichm\"uller equivalence classes $[f]$ of quasiconformal maps $f:\HH\to\HH$ that fix $0$, $1$ and $\infty$, and conjugate the action of $\Gamma$ onto another Fuchsian group $\Gamma_f$. In other words, $f$ descends to a quasiconformal map of $X=\HH /\Gamma$ onto $\HH /\Gamma_f$. Two quasiconformal maps $f,f_1:\HH\to\HH$ (which conjugate $\Gamma$ onto another Fuchsian group) are {\it Teichm\"uller equivalent} if they agree on the ideal boundary $\hat{\R}=\R\cup\{\infty\}$ of the upper half-plane $\HH$.

The Beltrami coefficient $\mu :=\frac{f_{\bar{z}}}{ f_z}$ of a quasiconformal map $f:\HH\to\HH$ satisfies $\|\mu\|_{\infty}<1$. Conversely, given $\mu\in L^{\infty}(\HH )$ with $\|\mu\|_{\infty}<1$ there exists a unique quasiconformal map $f:\HH \to\HH$ that fixes $0$, $1$ and $\infty$ whose Beltrami coefficient is $\mu$. A quasiconformal map $f$ conjugates $\Gamma$ onto another Fuchsian group if and only if $\mu\circ\gamma (z)\frac{\overline{\gamma'(z)}}{\gamma'(z)}=\mu (z)$ for almost all  $z\in\HH$ and for all $\gamma\in \Gamma$. 
Two Beltrami coefficients are {\it Teichm\"uller equivalent} if the corresponding normalized quasiconformal maps are equal on $\hat{\R}$.  Alternatively, $T (X)$ is the set of all Teichm\"uller classes $[\mu]$ of Beltrami coefficients that satisfy $\mu\circ\gamma (z)\frac{\overline{\gamma'(z)}}{\gamma'(z)}=\mu (z)$ for all $\gamma\in \Gamma$ and for almost all $z\in\HH$ (for example, see \cite{GardinerLakic}).

A quasiconformal map $f:\hat{\mathbb{C}}\to\hat{\mathbb{C}}$ that conjugates  $\Gamma<PSL_2(\mathbb{R})$ onto a subgroup of $PSL_2(\mathbb{C})$ fixing $0$, $1$ and $\infty$ represents an element of the quasi-Fuchsian space ${QF}(X )$, where $X=\HH /\Gamma$. The group $\Gamma_f:=f\Gamma f^{-1}$ is said to be a quasiFuchsian group. Two quasiconformal maps $f$ and $g$ fixing $0$, $1$ and $\infty$ that conjugate $\Gamma$ onto a subgroup of $PSL_2(\C )$ are {\it equivalent} if they agree on $\hat{\R}=\R\cup\{\infty\}$. In general, the quasiconformal maps $f$ representing a point in ${QF}(X )$ do not preserve $\hat{\R}$. Denote by $[f]\in {QF}(X )$ the corresponding equivalence class. Equivalently, we can define $ {QF}(X )$ to consist of all equivalence classes $[\mu ]$ of Beltrami coefficients $\mu$ on $\C$ satisfying  $\mu\circ\gamma (z)\frac{\overline{\gamma'(z)}}{\gamma'(z)}=\mu (z)$ for all $\gamma\in \Gamma$ and for almost all $z\in\C$ where two Beltrami coefficients are equivalent if their corresponding normalized quasiconformal maps agree on $\hat{\R }$ (and conjugate $\Gamma$ onto subgroups of $PSL_2(\mathbb{C})$). 

A quasiconformal map $f:\HH\to\HH$ extends by the reflection $f(z)=\overline{f(\bar{z})}$ for $z$ in the lower half-plane $\HH_{-}=\{ z:Im(z)<0\}$ to a quasiconformal map of $\hat{\C}$ that preserves $\hat{\mathbb{R}}$. The corresponding Beltrami coefficient satisfies $\mu (z)=\overline{\mu (\bar{z})}$ for $z$ in the lower half-plane $\HH_{-}$. The Teichm\"uller space $T (X)$ embeds into the quasi-Fuchsian space ${QF}(X)$ by extending each Beltrami coefficient $\mu$ on $\HH$ to $\C$ using the reflection in the real line. In this embedding, $T(X)$ is a totally real submanifold of $QF(X)$.

Bers introduced a complex Banach manifold structure to the Teichm\"uller space $T (X)$. The complex chart around the basepoint $[0 ]\in T (X)$ is obtained as follows. Let $\tilde{\mu}$ be the Beltrami coefficient which equals $\mu$ in the upper half-plane $\HH$ and equals zero in the lower half-plane $\HH_{-}$. The solution $f=f^{\tilde{\mu}}$ to the Beltrami equation $ f_{\bar{z}}=\tilde{\mu}  f_z$  is conformal in the lower half-plane $\HH_{-}$. The Schwarzian derivative 
$$
S(f^{\tilde{\mu}})(z)=\frac{(f^{\tilde{\mu}})'''(z)}{(f^{\tilde{\mu}})'(z)}-\frac{3}{2}\Big{(}\frac{(f^{\tilde{\mu}})''(z)}{(f^{\tilde{\mu}})'(z)}\Big{)}^2
$$
for $z\in\HH_{-}$ defines a holomorphic function $\phi (z)=S(f^{\tilde{\mu}})(z)$ which satisfies
$(\phi \circ\gamma )(z) \gamma'(z)^2=\phi (z)$ and $\|\varphi\|_{b}:=\sup_{z\in\HH_{-}}|y^2\phi (z)|<\infty$, called a {\it cusped form} for $X$. The space of all cusped forms $\phi :\HH_{-}\to\C$ for $X$ is a complex Banach space ${Q}_b(X)$ with the norm $\|\cdot\|_b$ (see \cite{GardinerLakic}).  

The Schwarzian derivative map is a holomorphic map from the unit ball in $L^{\infty}(\HH )$ onto an open subset of ${Q}_b(X)$ and it projects to a homeomorphism $\Phi$ from $T (X)$ to an open subset of ${Q}_b(X)$ containing the origin. 
The open ball $B_{[0]}(\frac{1}{2}\log 2)$ in $T (X)$ of radius $\frac{1}{2}\log 2$ and center $[0]$ maps under $\Phi$ onto an open set in ${Q}_b(X)$ which contains the ball of radius $\frac{2}{3}$ and is contained in the ball of radius $2$ with center $0\in{Q}_b(X)$. The map $\Phi :B_{[0]}(\frac{1}{2}\log 2)\to {Q}_b(X)$ is a chart map for the base point $[0]\in T (X)$ (see \cite[\S 6]{GardinerLakic}). The Ahlfors-Weill section provides an explicit formula for $\Phi^{-1}$ on the ball of radius $\frac{1}{2}$ and center $0$ in ${Q}_b(X)$. Namely if $\phi\in {Q}_b(X)$ with $\|\phi\|_{b}<\frac{1}{2}$ then Ahlfors and Weill prove that $\Phi^{-1}(\phi )=[-2y^2\phi (\bar{z})]$ (see \cite[\S 6]{GardinerLakic}). The Beltrami coefficient $\eta_{\phi}(z):=-2y^2\phi (\bar{z})$  is said to be {\it harmonic}. The Ahlfors-Weill formula gives an explicit expression of Beltrami coefficients that represent points in $T (X)$ corresponding to the holomorphic disks $\{\tau\phi :|\tau |<1,\|\phi\|_{b}<\frac{1}{2}\}$ in the chart in ${Q}_b(X)$, namely $$\Phi^{-1}(\{\tau\phi :|\tau |<1\})=\{ [\tau\eta_\phi ]\in T (X) :|\tau |<1\}.$$
More generally, for $\phi_0,\phi\in Q_b(X)$ with $\|\phi_0\|_b<1/2$, the holomorphic disk based at $\phi_0$ in the direction $\phi$ is given by  $$\{ \phi_0+\tau\phi :|\tau |<(1/2-\|\phi_0\|_b)/\|\phi\|_b \}.$$
We have
\begin{equation}
\label{eq:section-h-disk}
\Phi^{-1} (\{ \phi_0+\tau\phi :|\tau |<\frac{1/2-\|\phi_0\|_b}{\|\phi\|_b}\} )=\{ [\eta_{\phi_0}+\tau\eta_\phi ]\in T (X) :|\tau |<\frac{1/2-\|\phi_0\|_b}{\|\phi\|_b}\}
\end{equation}
 
Let $[\mu_0]\in T (X)$ be fixed and define $X^{[\mu_0]}=\mathbb{H}^2/ f^{\mu_0}\Gamma (f^{\mu_0})^{-1}$ to be the image Riemann surface. A chart for $[\mu_0]\in T (X)$ is given by the chart of $[0]\in T (X^{[\mu_0]})$ under the {\it translation map} $T_{[\mu_0 ]}:T (X)\to T (X^{[\mu_0]})$ defined by
$$
T_{[\mu_0 ]}([f])=[f\circ (f^{\mu_0})^{-1}]
$$
for $[f]\in T (X)$. Namely, let $B^{[\mu_0]}_{[0]}(\frac{1}{2}\log 2)\subset T (X^{[\mu_0]})$ be a ball with center $[0]$ and radius $\frac{1}{2}\log 2$ in the Teichm\"uller space of $X^{[\mu_0]}$. Let $\Phi_{[\mu_0]}$ be the Bers embedding for $T (X^{[\mu_0]})$. Then the chart for $[\mu_0]$ is given by
$$
\Phi_{[\mu_0]}\circ T_{[\mu_0]}:T_{[\mu_0]}^{-1}(B^{[\mu_0]}_{[0]}(\frac{1}{2}\log 2))\to {Q}_b(X^{[\mu_0]})
$$
where $T_{[\mu_0]}^{-1}(B^{[\mu_0]}_{[0]}(\frac{1}{2}\log 2))=B_{[\mu_0]}(\frac{1}{2}\log 2)$ because the translation map $T_{[\mu_0]}$ is an isometry for the Teichm\"uller metric. For $\phi_0,\phi\in Q_b(X^{[\mu_0]})$ with $\|\phi_0\|_b<1/2$ and for $\tau\in\mathbb{C}$ with $|\tau |<\frac{1/2-\|\phi_0\|_b}{\|\phi\|_b}$, using the Ahlfors-Weill section and the chain rule we have
\begin{equation}
\label{eq:chart-translation}
(\Phi_{[\mu_0]}\circ T_{[\mu_0]})^{-1}(\phi_0+\tau\phi )=\Big{[}\frac{ (\eta_{\phi_0}+\tau\eta_\phi )\circ f^{\mu_0}
\frac{\overline{f^{\mu_0}_{{z}}}}{f^{\mu_0}_z}+{\mu_0 }}{1+ (\eta_{\phi_0}+\tau\eta_\phi )\circ f^{\mu_0}\frac{\overline{f^{\mu_0}_{{z}}}}{f^{\mu_0}_z}\overline{\mu_0}}\Big{]}.
\end{equation}

By the Bers simultaneous uniformization theorem, the quasi-Fuchsian space ${QF}(X)$ is identified with the product of the Teichm\"uller space $T (X)$ of the Riemann surface $X$ and the Teichm\"uller space $T (\bar{X})$ of the mirror image Riemann surface $\bar{X}$. A point $[\mu ]\in {QF}(X)$ is represented by a Beltrami coefficient $\mu$ on $\C$ and we identify it with the pair of (equivalence classes of) Beltrami coefficients $([\mu_1],[\mu_2])$, where $\mu_1=\mu |_{\HH}$ and $\mu_2=\mu |_{\HH_{-}}$.  The Teichm\"uller space $T (\bar{X})$ is defined as the equivalence classes of Beltrami coefficients in $\HH_{-}$ invariant under $\Gamma$ where two Beltrami coefficients are equivalent if their corresponding normalized quasiconformal maps from $\HH_{-}$ onto itself agree on $\hat{\R}$. The charts on ${QF}(X)$ are given by the product charts on $T (X)\times T (\bar{X})$ and the inverse of the charts are given by the product of the Ahlfors-Weil sections. The Teichm\"uller space $T(X )$ embeds into ${QF}(X)$ as a totally real-analytic submanifold by the formula $[\mu ]\mapsto [\tilde{\mu}]$, where $\tilde{\mu}(z)=\mu (z)$ for $z\in\HH$ and $\tilde{\mu }(z)=\overline{\mu (\bar{z})}$ for $z\in\HH_{-}$. Note that the correspondence $[\mu ]\mapsto [\tilde{\mu}]$ is real-linear but not complex-linear.

Given a pair of Beltrami coefficients $(\mu_1 ,\mu_2 )$ with $\mu_1\in L^{\infty}(\HH )$ and $\mu_2\in L^{\infty}(\HH_{-})$ define $\tilde{\mu}_1=\mu_1$ on $\HH$, $\tilde{\mu}_1=0$ on $\HH_{-}$, and $\tilde{\mu}_2=\mu_2$ on $\HH_{-}$, $\tilde{\mu}_2=0$ on $\HH$. The Schwarzian derivatives $S(f^{\tilde{\mu}_1})$ and  $S(f^{\tilde{\mu}_2})$ give a pair of holomorphic functions $(\phi_1,\phi_2)$ in $\HH_{-}$ and $\HH$, respectively. The holomorphic functions $\phi_1$ and $\phi_2$ (in $Q_b(X)$ and $Q_b(\bar{X})$) satisfy the invariance property under $\Gamma$ and the boundedness in the $\|\cdot\|_b
$-norm in their respective domains. Let $\eta_i$ be the harmonic Beltrami coefficient corresponding to $\phi_i$ under the Ahlfors-Weill section. Then the pair of harmonic Beltrami coefficients $\eta:=(\eta_1,\eta_2)$ represents the point in ${QF}(X)$ that is mapped to $\phi :=(\phi_1,\phi_2)$ under the Bers embedding. Such $\eta$ is called a {\it complex harmonic differential}, giving the inverse of the chart map of $QF(X)$. 

Let $[{\mu_0}]\in T(X)\subset QF(X)$ with ${\mu_0}^+:={\mu_0}|\mathbb{H}^2$ and ${\mu_0}^-:={\mu_0}|\mathbb{H}^2_-$. 
If $\phi_0^+,\phi^+\in Q_b(X^{[{\mu_0}^+]})$, $\phi_0^-,\phi^-\in Q_b(\bar{X}^{[{\mu_0}^-]})$ with $\|\phi_0^+\|_b<1/2$, $\|\phi_0^-\|_b<1/2$, the holomorphic disk based at $(\phi_0^+,\phi_0^-)$ in the direction $(\phi^+,\phi^-)$ is given by  
$$\{ (\phi_0^++\tau\phi^+,\phi_0^-+\tau\phi^-) :|\tau |<\min (\frac{1/2-\|\phi_0^+\|_b}{\|\phi^+\|_b},\frac{1/2-\|\phi_0^-\|_b}{\|\phi^-\|_b})  \}.$$ 
The formulas (\ref{eq:section-h-disk}) and (\ref{eq:chart-translation}) extend to the holomorphic disk above since they are defined separately on $\mathbb{H}^2$ and $\mathbb{H}^2_-$. We note that the corresponding Beltrami coefficient 
\begin{equation}
\label{eq:hol-disk-qf}
\mu (\tau ):=((\Phi_{[\mu_0^+]}\circ T_{[\mu_0^+]})^{-1}(\phi_0^++\tau\phi^+),(\Phi_{[\mu_0^-]}\circ T_{[\mu_0^-]})^{-1}(\phi_0^-+\tau\phi^-))  
 \end{equation}
 is holomorphic in $\tau$ as a function in $L^{\infty}(\mathbb{C})$.

\section{The shear coordinates parametrizations of the Teichm\"uller space of Riemann surfaces with
bounded ideal triangulations}

Let $X$ be a Riemann surface with a Fuchsian covering group of the first kind and countably many punctures that accumulate to each topological end of $X$. An ideal triangulation $\mathcal{T}$ is said to be {\it bounded} if the edges of $\mathcal{T}$ connect punctures, there is an upper bound on the number of incident edges on every puncture, and the shears on all edges are equal to zero. Let $d>0$ be the least common multiple on the number of edges of $\mathcal{T}$ that are incident to the punctures of $X$.

Let $\mathcal{E}$ be the edges of the ideal triangulation $\mathcal{T}$ of $X$.
We define $\ell^{\infty}(\mathcal{E},\mathbb{R})$ to be the vector space of all bounded functions $f:\mathcal{E}\to \mathbb{R}$ with the supremum norm $\| f\|_{\infty}=\sup_{e\in\mathcal{E}}|f(e)|$.

We give a parametrization of the Teichm\"uller space $T(X)$ in terms of the shear coordinates on the edges $\mathcal{E}$. Note that this parametrization complements the parametrization of $T(X)$ when $X$ has a bounded pants decomposition (see Shiga \cite{Shiga}) and when $X$ has an upper bounded pants decomposition (see Alessandrini-Liu-Papadopoulos-Su \cite{ALPS}). Indeed, in the case we consider any pants decomposition has unbounded cuff lengths (see \cite{Kinjo,BPV}).

\begin{thm}
\label{thm:Teich-param}
Let $X$ be a Riemann surface with bounded ideal triangulation $\mathcal{T}$. Then, the shear map 
$$
s_{*}:T(X)\to\ell^{\infty}(\mathcal{E},\mathbb{R})
$$
is a homeomorphism onto its image, and the image $s_*(T(X))$ consists of all $s\in \ell^{\infty}(\mathcal{E},\mathbb{R})$ such that the sum of shears at each puncture equals zero.
\end{thm}

\begin{proof}
We first characterize the image $s_*(T(X))$ of $T(X)$. For each $[f]\in T(X)$, the punctures on $X$ correspond to the punctures on the image surface $f(X)$. Therefore, the sum of the shears $s_*([f])$ on the edges ending at a puncture is zero. Moreover, the shear function $s_*([f]):\mathcal{E}\to\mathbb{R}$ is bounded since each quadruple with zero shears is mapped by a quasiconformal map to a quadruple whose shear is bounded with a bound depending only on the quasiconformal constant of $f$ (see \cite[Theorem A]{Saric10}, \cite[Theorem 1]{Saric23} and \cite[Theorem 1.1]{PS}). 

Conversely, 
assume that $s\in\ell^{\infty}(\mathcal{E},\mathbb{R})$ and that the sum of the shears on the edges at each puncture is zero (when projected to $X$). We need to prove that $s$ is in $s_*(T(X))$. Note that any $s:\mathcal{E}\to\mathbb{R}$ defines a developing map $h_s$ from the set of vertices $\hat{\mathbb{Q}}$ into $\hat{\mathbb{R}}$ that realizes $s$ (see Penner \cite{Penner} and also \cite{Saric10}). By \cite[Theorem 1.1]{PS}, the developing map $h_s$ is quasisymmetric if there exists $C>0$ such that at every fan of edges $\{ e_n\}_{n\in\mathbb{Z}}$, for each $m\in\mathbb{Z}$ and each $k\in\mathbb{N}$, we have
\begin{equation}
\label{eq:qs-PS}
|\sum_{j=m}^{m+k} s(e_j)|<C.
\end{equation}
We will use the above sufficient condition for quasisymmetry. Indeed, we have that
$$
\sum_{j=m}^{m+q-1} s(e_j)=0
$$
for all fans and all starting edges $e_m$, in the fans. Therefore, the condition in (\ref{eq:qs-PS}) is satisfied for $C=(d-1)\| s\|_{\infty}$, where $d$ is the least common multiple of the number of incidence geodesics at each puncture of $X$. The map $h_s$ is quasisymmetric,. 

Next, we prove that the map $s_*:T(X)\to\ell^{\infty}(\mathcal{E},\mathbb{R})$ is continuous. Since both $T(X)$ and $\ell^{\infty}(\mathcal{E},\mathbb{R})$ are metric spaces, it is enough to prove the sequential continuity. Assume that a sequence of quasisymmetric maps $h_n:\hat{\mathbb{R}}\to \hat{\mathbb{R}}$ that fix $0$, $1$ and $\infty$ converge to a quasisymmetric map $h$ in the Teichm\"uller topology. If $s_n$ and $s$ are the shear functions corresponding to $h_n$ and $h$, then by \cite{Saric10}, we have that $\sup_{p\in\hat{\mathbb{Q}}}M_{s_n,s}(p)\to 1$ as $n\to\infty$. In particular, this implies that, as $n\to\infty$,
$$
\sup_{e\in\mathcal{E}}e^{|s_n(e)-s(e)|}\to 1
$$
which is equivalent to
$$
\| s_n-s\|_{\infty}\to 0.
$$

Now we prove that $(s_*)^{-1}:s_*(T(X))\to T(X)$ is continuous. Let $\| s_n-s\|_{\infty}\to 0$ as $n\to\infty$. If $\sup_{p\in\hat{\mathbb{Q}}} M_{s_n,s}(p)\to 1$ as $n\to\infty$ then $h_{s_n}$ converges to $h_s$ in the Teichm\"uller topology by \cite{Saric10}. Therefore, we need to prove that $\sup_{p\in\hat{\mathbb{Q}}} M_{s_n,s}(p)\to 1$ as $n\to\infty$. 

Given a vertex $p\in\hat{\mathbb{Q}}$, denote by $\{ e_j\}_{j\in\mathbb{Z}}$ the fan of geodesics in $\mathcal{F}$ with one endpoint $p$. Let $k\in\mathbb{Z}$ and $m\in \mathbb{N}$. We have
$$s(p;m,k)=e^{s_k}\frac{1+e^{s_{k+1}}+e^{s_{k+1}+s_{k+2}}+...+e^{s_{k+1}+s_{k+2}+...+s_{k+m}}}{1+e^{-s_{k-1}}+e^{-s_{k-1}-s_{k-2}}+...+e^{-s_{k+1}-s_{k-2}-...-s_{k-m}}}$$
and 
$$
M_{s_n,s}(p)=\sup_{m,k}\frac{s_n(p;m,k)}{s(p;m,k)}.
$$

Let $m=qd+r$ for some $q,r\in \mathbb{N}\cup\{ 0\}$ with $0\leq r<d$. Since the sum of shears about any puncture is zero, any sum of shears on $d$ adjacent geodesics in a fan is zero. Thus $$s(p;m,k)=e^{s_k}\frac{q(1+e^{s_{k+1}}+...+e^{s_{k+1}+s_{k+2}+...+s_{k+d-1}})+e^{s_{k+1}}+...+e^{s_{k+1}+s_{k+2}+...+s_{k+r}}}{q(1+e^{-s_{k-1}}+...+e^{-s_{k-1}-s_{k-2}-...-s_{k-d+1}})+e^{-s_{k-1}}+...e^{-s_{k-1}-s_{k-2}-...-s_{k-r}}}$$

Let $C>0$ be such that
$$
\| s_n\|_{\infty},\| s\|_{\infty} \leq C
$$
for all $n$. Consider $\frac{s(p;m,k)}{s_n(p;m,k)}$ and note that it can be written as a product of three factors $$\frac{1+e^{s_{k+1}}+...+e^{s_{k+1}+s_{k+2}+...+s_{k+m}}}{1+e^{(s_n)_{k+1}}+...+e^{(s_n)_{k+1}+(s_n)_{k+2}+...(s_n)_{k+m}}},$$
$$
e^{s_k-(s_n)_k}
$$
and
$$
\frac{1+e^{-(s_n)_{k-1}}+...+e^{-(s_n)_{k-1}-(s_n)_{k-2}-...-(s_n)_{k-m}}}{1+e^{-s_{k+1}}+...+e^{-s_{k+1}-s_{k+2}-...-s_{k+m}}},
$$
where $s_k:=s(e_k)$ and $(s_n)_k:=s_n(e_k)$.
The above middle factor satisfies $
e^{-\| s-s_n\|_{\infty}}\leq e^{s_k-(s_n)_k}\leq e^{\| s-s_n\|_{\infty}}
$ which implies that it converges to $1$ uniformly in all $k$ and all fans. 

To estimate the first factor, consider
\begin{equation}
\label{eq:numerator}
|\frac{1+e^{s_{k+1}}+...+e^{s_{k+1}+s_{k+2}+...+s_{k+m}}}{1+e^{(s_n)_{k+1}}+...+e^{(s_n)_{k+1}+(s_n)_{k+2}+...(s_n)_{k+m}}}-1|.
\end{equation}
Once we write the above expression as a single fraction and use the fact that the sum of $d$ consecutive shears is zero with $m=qd+r$, the numerator is bounded above by the sum of the terms
$$
(q+1)|e^{\sum_{i=1}^ls_{k+i}}-e^{\sum_{i=1}^l(s_n)_{k+i}}|
$$ 
for $l=1,\ldots ,m$. 

Note that
\begin{equation}
\label{eq:insert}
\begin{split}
|e^{\sum_{i=1}^ls_{k+i}}-e^{\sum_{i=1}^l(s_n)_{k+i}}|\leq |e^{\sum_{i=1}^ls_{k+i}}-e^{(\sum_{i=1}^{l-1}(s_n)_{k+i})+s_l}|+\ \ \ \ \ \\  | e^{(\sum_{i=1}^{l-1}(s_n)_{k+i})+s_l}-e^{\sum_{i=1}^l(s_n)_{k+i}}|
\end{split}
\end{equation}
By inserting appropriate terms as in (\ref{eq:insert}), we obtain
\begin{equation}
\begin{split}
|e^{\sum_{i=1}^ls_{k+i}}-e^{(\sum_{i=1}^{l-1}(s_n)_{k+i})+s_l}|\leq e^{C} |e^{\sum_{i=1}^{l-1}s_{k+i}}-e^{(\sum_{i=1}^{l-1}(s_n)_{k+i})}|\leq\ldots \\
\leq 2^de^{dC}(e^{\| s_n-s\|_{\infty}}-1).
\end{split}
\end{equation}
Therefore, there is a constant $K=K(d)$ such that the numerator in (\ref{eq:numerator}) is bounded from the above by $(q+1)K(d)\| s-s_n\|_{\infty}$. The denominator in (\ref{eq:numerator}) is bounded from below by $q K_1+K_2r$, where both $K_1$ and $K_2$ are positive when $\| s-s_n\|_{\infty}$ is smaller than some fixed positive number. Therefore we have
$$
|\frac{1+e^{s_{k+1}}+...+e^{s_{k+1}+s_{k+2}+...+s_{k+m}}}{1+e^{(s_n)_{k+1}}+...+e^{(s_n)_{k+1}+(s_n)_{k+2}+...(s_n)_{k+m}}}-1|\leq K\| s-s_n\|_{\infty}
$$
for a fixed $K>0$ and $\| s-s_n\|_{\infty}$ bounded from the above.
 Analogously, the third factor also satisfies
 $$
|\frac{1+e^{-(s_n)_{k-1}}+...+e^{-(s_n)_{k-1}-(s_n)_{k-2}-...-(s_n)_{k-m}}}{1+e^{-s_{k+1}}+...+e^{-s_{k+1}-s_{k+2}-...-s_{k+m}}}-1|\leq K\| s-s_n\|_{\infty} .
$$

Putting the three estimates together, we conclude that
$$
\sup_p M_{s,s_n}(p)\to 1
$$
as $n\to \infty$. Therefore $h_{s_n}\to h_s$ as $n\to\infty$ in the Teichm\"uller topology. Therefore $(s_*)^{-1}$ is continuous, which finishes the proof of the theorem.
\end{proof}

\begin{rem}
The above theorem established that the $\ell^{\infty}$-norm describes the correct topology on the shears, making them homeomorphic to the Teichm\"uller space. This fact depends on the existence of bounded ideal triangulation. For example, the $\ell^{\infty}$-topology on the shear function on the Farey triangulation corresponding to the universal Teichm\"uller space is not homeomorphic to the Teichm\"uller topology. It would be interesting to understand better the topology on shears for non-bounded ideal triangulations and, in particular, the case of the universal Teichm\"uller space (see also an open question in \cite{PS}).
\end{rem}

\section{The smoothness of shear coordinates}

Let $X=\mathbb{H}^2/\Gamma$ be a planar Riemann surface with an ideal triangulation $\mathcal{T}$ such that the number of edges at each puncture is bounded above and all the shears are zero. The triangulation $\mathcal{T}$ lifts to the Farey triangulation $\mathcal{F}$ of $\mathbb{H}^2$ and the quotient by $\Gamma$ recovers $\mathcal{T}$.

In this section, we prove that the shear coordinates parametrization of the Teichm\"uller space $T(X)$ 
is a real-analytic diffeomorphism. To do so, we extend the shears parametrization from a neighborhood of $T(X)$ in the quasi-Fuchsian space $QF(X)$ to complex shears and show that the extended map is a biholomorphic local homeomorphism.
 
\subsection{The complexification of shears}

Let $(a,b,c,d)$ be a quadruple of distinct points on $\hat{\mathbb{R}}$ given in the positive direction for its orientation as the boundary of the upper half-plane $\mathbb{H}^2$. Let $\Delta_1$ be the ideal triangle in $\mathbb{H}^2$ with endpoints $(a,b,d)$ and let $\Delta_2$ be the ideal triangle with endpoints $(b,c,d)$. Recall that the shear of the common side $(b,d)$ is given by
$$
s=\log cr(a,b,c,d)=\log \frac{(d-a)(c-b)}{(d-c)(b-a)}.
$$
Note that $cr(a,b,c,d)\notin \hat{\mathbb{C}}\setminus\{ -1,0,\infty\} =\mathbb{C}\setminus\{ -1,0\}$ as long as $(a,b,c,d)$ are distinct points in $\hat{\mathbb{C}}$. In fact, $cr(a,b,c,d)=-1$ if either $c=a$ or $b=d$; $cr(a,b,c,d)=0$ if either $d=a$ or $c=b$; $cr(a,b,c,d)=\infty$ if either $c=d$ or $b=a$.

Let $\mu_0$ be a Beltrami coefficient in $\mathbb{C}$ that represents a point $[f^{\mu_0}]\in T(X)$. Thus $\mu_0$ satisfies
$$
\mu_0(\gamma (z))\frac{\overline{\gamma'(z)}}{\gamma'(z)}=\mu_0(z)
$$
for all $\gamma\in\Gamma$, and 
$$
\mu_0(z)=\overline{\mu_0(\bar{z})}
$$
for a.a. $z\in\mathbb{C}$.

Let $\mathcal{E}=E(\mathcal{F})$ be the set of edges of the Farey triangulation $\mathcal{F}$. The shear coordinates $s_{[\mu_0]}:\mathcal{E}\to\mathbb{R}$ of the point $[f^{\mu_0}]\in T(X)$ are given by
$$
s_{[\mu_0]}(e)=\log cr(f^{\mu_0}(a),f^{\mu_0}(b),f^{\mu_0}(c),f^{\mu_0}(d)),
$$
where $(a,b,c,d)\in\hat{\mathbb{R}}$ is the quadruple of points representing the two adjacent triangles of $\mathcal{F}$ with common edge $e$.

Since $f^{\mu_0}$ is a quasiconformal map, it quasi-preserves the moduli of quadrilaterals (see \cite{GardinerLakic}). Since the quadrilateral $\mathbb{H}^2(a,b,c,d)$ whose interior is $\mathbb{H}^2$ and vertices are $(a,b,c,d)$ has modulus $1$, it follows that the image quadrilateral $\mathbb{H}^2(f^{\mu_0}(a),f^{\mu_0}(b),f^{\mu_0}(c),f^{\mu_0}(d))$ has modulus between $1/K$ and $K$, where $K=\frac{1+\|\mu_0\|_{\infty}}{1-\|\mu_0\|_{\infty}}$. Since the moduli of the quadrilaterals and the cross-ratios of the quadruples are continuous functions of the vertices, it follows that they are continuous functions of each other (see \cite{LV}).  Therefore, there exists $x_0>1$ such that 
$$1/x_0\leq cr(f^{\mu_0}(a),f^{\mu_0}(b),f^{\mu_0}(c),f^{\mu_0}(d))\leq x_0$$ 
for all quadruples $(a,b,c,d)$ corresponding to the edges $e\in\mathcal{E}$. 

We denote by $\rho_{-1,0}$ the hyperbolic metric on the domain $\mathbb{C}\setminus\{ -1,0 \}=\hat{\mathbb{C}}\setminus\{ -1,0,\infty\}$. Let $[1/x_0,x_0]$ be the closed interval in $\mathbb{R}$ with endpoints $1/x_0$ and $x_0$ which we consider as a subset of $\mathbb{C}\setminus\{ -1,0 \}$. Let $r_0>0$ be a fixed constant such that the set of points $D_{r_0}([1/x_0,x_0])$ that are on the hyperbolic distance at most $r_0$ from $[1/x_0,x_0]$ lies in the right half-plane in $\mathbb{C}$, i.e., $D_{r_0}([1/x_0,x_0])\subset\{ z\in\mathbb{C}:Re(z)>0\}$. This property will allow us to define the complex shears for quasiFuchsian deformations of $X$ since taking the logarithm of the cross-ratios in $\{ z\in\mathbb{C}:Re(z)>0\}$ gives the imaginary parts in the interval $[-\pi /2 ,\pi /2]$. 

\begin{lem}
\label{lem:ext-U}
Let $[f^{\mu_0}]\in T(X)$ be fixed and let $r_0>0$ be such that $D_{r_0}([1/x_0,x_0])\subset\{ z\in\mathbb{C}:Re(z)>0\}$  as above. Then there exists $\epsilon_0>0$ such that
$$
cr(f^{\mu_0+\nu}(a),f^{\mu_0+\nu}(b),f^{\mu_0+\nu}(c),f^{\mu_0+\nu}(d))\in D_{r_0}([1/x_0,x_0])
$$
for all $\nu\in L^{\infty}(\mathbb{C})$ with $\|\nu\|_{\infty}<\epsilon_0$ and for all quadruples $(a,b,c,d)$ that correspond to edges $e\in\mathcal{E}$.
\end{lem}

\begin{proof}
Let $\mu\in L^{\infty}(\mathbb{C})$ with $\|\mu\|_{\infty}=1$. The map $\tau\mapsto cr(f^{\mu_0+\tau\mu}(a,b,c,d))$ is well-defined and holomorphic for all complex numbers $\tau$ with $|\tau |<1-\|\mu_0\|_{\infty}$ because, for such that $\tau$, $\|\mu_0+\tau\mu\|_{\infty}<1$. In addition, since $f^{\mu_0+\tau\mu}$ is quasiconformal, the image quadruple $f^{\mu_0+\tau\mu}(a,b,c,d)$ has pairwise distinct components which implies that $cr(f^{\mu_0+\tau\mu}(a,b,c,d))\in \mathbb{C}\setminus\{ -1,0\}$.  

Therefore, we have a holomorphic map from $\mathbb{D}_{1-\|\mu_0\|_{\infty}}=\{ \tau\in\mathbb{C}:|\tau |< 1-\|\mu_0\|_{\infty}\}$ into $\mathbb{C}\setminus\{ -1,0\}$. By Schwarz's Lemma (see \cite{Ahlfors}), the holomorphic map is distance non-increasing for the hyperbolic metrics on the domain and range. Since $\tau =0\mapsto cr(f^{\mu_0}(a,b,c,d))\subset [1/x_0,x_0]$, there exists $\delta >0$ such that $cr(f^{\mu_0+\tau\mu}(a,b,c,d))\subset D_{r_0}([1/x_0,x_0])$ for all $\tau$ on the hyperbolic distance less than $\delta$ from $0$ for the domain $\mathbb{D}_{1-\|\mu_0\|_{\infty}}$. The choice of $\delta$ is independent of $\|\mu\|_{\infty}=1$ and of $(a,b,c,d)$ for $cr(a,b,c,d)=1$. The hyperbolic disk in the domain $\mathbb{D}_{1-\|\mu_0\|_{\infty}}$ centered in $0$ of radius $\delta$ is a Euclidean disk of some radius $\epsilon_0>0$ which finishes the proof.
\end{proof}

Given $[f^{\mu_0}]\in T(X)$, let $\epsilon_0>0$ be the constant from Lemma \ref{lem:ext-U}. Define
$U_{[f^{\mu_0}]}(\epsilon_0)\subset QF(X)$ to consist of all classes of quasiconformal maps $f^{\mu_0+\nu}$ with $\nu\in L^{\infty}(\mathbb{C})$, $\|\nu\|_{\infty}<\epsilon_0$, and
$$
\nu (\gamma (z))\frac{\overline{\gamma'(z)}}{\gamma'(z)}=\nu (z)
$$
for all $\gamma\in\Gamma$ and a.a. $z\in\mathbb{C}$. The set $U_{[f^{\mu_0}]}(\epsilon_0)$ contains an open neighborhood of $[f^{\mu_0}]$ in $QF(X)$. 

To $[f^{\mu_0+\nu}]\in U_{[f^{\mu_0}]}(\epsilon_0)\subset QF(X)$, we assign complex shear coordinates on the edges $\mathcal{E}$ of the Farey triangulation $\mathcal{F}$. Given $e\in\mathcal{E}$, let $(a,b,c,d)$ be the quadruple consisting of vertices of the two ideal hyperbolic triangles of $\mathcal{F}$ that share edge $e$ with $(b,c)$ the endpoints of $e$. Define
$$
s_{[f^{\mu_0+\nu}]}(e)=\log cr(f^{\mu_0+\nu}(a),f^{\mu_0+\nu}(b),f^{\mu_0+\nu}(c),f^{\mu_0+\nu}(d)),
$$
where the branch of the logarithm is chosen such that its imaginary parts are in the interval $[-\pi /2,\pi /2]$. The shear coordinate function $s_{*}: U_{[f^{\mu_0}]}(\epsilon_0)\to
(\mathbb{R}\times [-\pi /2,\pi /2])^{\mathcal{E}}$ is injective and, for each $e\in\mathcal{E}$,  $\tau\mapsto s_{[f^{\mu_0+\tau\nu}]}(e)$ is holomorphic in $\tau$.

\subsection{The map from $U_{[f^{\mu_0}]}(\epsilon_0)$ to the complex shear coordinates is holomorphic} 

Our goal in this subsection is to prove that $s_{*}: U_{[f^{\mu_0}]}(\epsilon_0)\to
(\mathbb{R}\times [-\pi /2,\pi /2])^{\mathcal{E}}$ is holomorphic, which will imply that its restriction to the Teichm\"uller space is a real-analytic map. Define $\ell^{\infty}(\mathcal{E},\mathbb{C})$ to be the complex Banach space of all bounded functions $g:\mathcal{E}\to\mathbb{C}$ with the norm $\| g\|_{\infty}=\sup_{e\in\mathcal{E}}|g(e)|$. 
 Since $cr(f^{\mu_0+\tau\mu}(a,b,c,d)\subset D_{r_0}([1/x_0,x_0])$, the image $s_{*}(U_{[f^{\mu_0}]}(\epsilon_0))$ is in $\log (D_{r_0}([1/x_0,x_0] )$ which means that $s_{[f^{\mu_0+\nu}]}\in \ell^{\infty}(\mathcal{E},\mathbb{C})$ for each $[f^{\mu_0+\nu}]\in B_{[f^{\mu_0}]}(\epsilon_0)$ and that $s_{*}(B_{[f^{\mu_0}]}(\epsilon_0))$ is a bounded subset of $\ell^{\infty}(\mathcal{E},\mathbb{C})$.

By \cite[Theorem 14.9, page 198]{Chae}, the map $s_{*}:U_{[f^{\mu_0}]}(\epsilon_0)\to \ell^{\infty}(\mathcal{E},\mathbb{C})$ is holomorphic if it is bounded and weakly holomorphic. We established above that 
$s_{*}$ is bounded.
Let $[\widetilde{\mu_0}]$ be an arbitrary point in $U_{[f^{\mu_0}]}$( $\epsilon_0$) which corresponds to $\phi_0=(\phi_0^+,\phi_0^-)$, $\phi_0^+ \in Q_b(X^{[\mu_0^+]})$ and $\phi_0^- \in Q_b(X^{[\mu_0^-]})$. Namely
$$
\widetilde{\mu_0}=((\Phi_{[\mu_0^+]}\circ T_{[\mu_0^+]})^{-1}(\phi_0^+),(\Phi_{[\mu_0^-]}\circ T_{[\mu_0^-]})^{-1}(\phi_0^-)).
$$ 
Let $\{ (\phi_0^++\tau\phi^+,\phi_0^-+\tau\phi^-) :|\tau |<\min (\frac{1/2-\|\phi_0^+\|_b}{\|\phi^+\|_b},\frac{1/2-\|\phi_0^-\|_b}{\|\phi^-\|_b})  \}$ be a disk centered at $\phi_0$ in the direction $\phi =(\phi^+,\phi^-)$.  The map
$$
\tau\mapsto \mu (\tau ):=((\Phi_{[\mu_0^+]}\circ T_{[\mu_0^+]})^{-1}(\phi_0^++\tau\phi^+),(\Phi_{[\mu_0^-]}\circ T_{[\mu_0^-]})^{-1}(\phi_0^-+\tau\phi^-))
$$
is holomorphic function of $\tau$ into $L^{\infty}(\mathbb{C})$.
To show that $s_{*}$ is weakly holomorphic, we need to 
prove that $\tau\mapsto s_{[f^{\mu (\tau )}]}$ is differentiable as a function of a single complex variable $\tau$ into the infinite-dimensional complex Banach space $\ell^{\infty}(\mathcal{E},\mathbb{C})$. 

\begin{prop}
\label{prop:diff-tau}
Let $[f^{\mu_0}]\in T(X)$ and let $U_{[f^{\mu_0}]}(\epsilon_0)$ be its neighborhood in $QF(X)$ as above. Let $\{ (\phi_0^++\tau\phi^+,\phi_0^-+\tau\phi^-) :|\tau |<\min (\frac{1/2-\|\phi_0^+\|_b}{\|\phi^+\|_b},\frac{1/2-\|\phi_0^-\|_b}{\|\phi^-\|_b})  \}$ be a complex disk in the image of  $U_{[f^{\mu_0}]}(\epsilon_0)$ under the chart map. Then the complex shear map from $\{\tau\in\mathbb{C}: |\tau |<\min (\frac{1/2-\|\phi_0^+\|_b}{\|\phi^+\|_b},\frac{1/2-\|\phi_0^-\|_b}{\|\phi^-\|_b})\}$ into $\ell^{\infty}(\mathcal{E},\mathbb{C})$ given by
$$
\tau\mapsto s_{[f^{\mu (\tau )}]}
$$
is holomorphic, where $\mu (\tau)=((\Phi_{[\mu_0^+]}\circ T_{[\mu_0^+]})^{-1}(\phi_0^++\tau\phi^+),(\Phi_{[\mu_0^-]}\circ T_{[\mu_0^-]})^{-1}(\phi_0^-+\tau\phi^-))$ .
\end{prop}

\begin{proof} Let $\widetilde{\mu_0}=((\Phi_{[\mu_0^+]}\circ T_{[\mu_0^+]})^{-1}(\phi_0^+),(\Phi_{[\mu_0^-]}\circ T_{[\mu_0^-]})^{-1}(\phi_0^-))$. 
By the chain rule (for example, see \cite[page 11]{GardinerLakic}), the Beltrami coefficient $\sigma (\tau )$ of $f^{\mu (\tau)}\circ (f^{\widetilde{\mu_0}})^{-1}$ is 
$$\sigma=\Big{(}\frac{\mu (\tau )-\widetilde{\mu_0}}{1-\overline{\widetilde{\mu_0}}\mu (\tau )}\frac{1}{\theta}\Big{)}\circ (f^{\widetilde{\mu_0}})^{-1},$$
where $\|\theta\|_{\infty}=1$.
Note that $\sigma =\sigma (\tau )$ is a holomorphic function of $\tau$ and $\sigma (0)=0$.

By the solution of the measurable Riemann's mapping theorem (see \cite{AhlforsBers}), the quasiconformal mapping $f^{\sigma}= f^{\mu (\tau)}\circ (f^{\widetilde{\mu_0}})^{-1}$ depends holomorphically on $\tau$ and it equals the identity when $\tau =0$. Here, we consider $f^{\sigma}$ to belong to the function space that has finite Dirichlet norms for $|z|<R$ and has  $\partial /\partial z$ and $\partial /\partial \bar{z}$ derivatives with finite $L^{\infty}$-norms. In particular, we conclude that $\sigma\mapsto f^{\sigma}$ is a holomorphic map in the parameter $\tau$ into the space of $L^{\infty}$-functions on $|z|<R$. Moreover, the map is bounded. Since a map is holomorphic if and only if it is   bounded and weakly holomorphic (see \cite[Theorem 14.9, page 198]{Chae}), the limit
\begin{equation}
\label{eq:uniform-hol}
\lim_{\omega\to 0}\frac{|f^{\sigma (\tau+\omega )}(z)-f^{\sigma (\tau )}(z)-\frac{d}{d\tau}f^{\sigma (\tau )}(z)|}{|\omega |}=0
\end{equation} 
is uniform in all holomorphic families of Beltrami coefficients $\sigma (\tau )$ if $\| \frac{d}{d\tau}\sigma (\tau )\|_{\infty}\leq C$ for some fixed $C$ (see \cite{AhlforsBers}).

The quadruple $(a,b,c,d)$ corresponding to an edge $e\in\mathcal{E}$ can be outside $|z|<R$, and the above estimate cannot be applied directly.  Let $A\in PSL_2(\mathbb{C})$ be such that it maps $(f^{\widetilde{\mu_0}}(a),f^{\widetilde{\mu_0}}(b),f^{\widetilde{\mu_0}}(d))$ onto $(0,1,\infty )$. 
Since the quadruple $(f^{\widetilde{\mu_0}}(a),f^{\widetilde{\mu_0}}(b),f^{\widetilde{\mu_0}}(c),f^{\widetilde{\mu_0}}(d))$ has cross-ratio in the set $D_{r_0}([1/x_0,x_0])$, it follows that $|A(f^{\widetilde{\mu_0}}(c))|<R(x_0)$ for some $R=R(x_0)$ which only depends of $x_0$. We write
$$
f^{\mu (\tau )}=(f^{\sigma (\tau )}\circ A^{-1})\circ (A\circ f^{\widetilde{\mu_0}})
$$
and note that the norm of the Beltrami coefficient of $f^{\sigma (\tau )}\circ A^{-1}$ is equal to $\|\sigma\|_{\infty}$. 

Since
$$
cr(f^{\widetilde{\mu_0}}(a),f^{\widetilde{\mu_0}}(b),f^{\widetilde{\mu_0}}(c),f^{\widetilde{\mu_0}}(d))=cr(A\circ f^{\widetilde{\mu_0}}(a),A\circ f^{\widetilde{\mu_0}}(b),A\circ f^{\widetilde{\mu_0}}(c),A\circ f^{\widetilde{\mu_0}}(d))
$$
and
$$
cr(A\circ f^{\widetilde{\mu_0}}(a),A\circ f^{\widetilde{\mu_0}}(b),A\circ f^{\widetilde{\mu_0}}(c),A\circ f^{\widetilde{\mu_0}}(d))=A\circ f^{\widetilde{\mu_0}}(c)-1
$$
it follows that the limit in (\ref{eq:uniform-hol}) is uniform over all $cr(f^{\widetilde{\mu_0}}(a),f^{\widetilde{\mu_0}}(b),f^{\widetilde{\mu_0}}(c),f^{\widetilde{\mu_0}}(d))$ corresponding to all quadruples $(a,b,c,d)$ representing the edges $\mathcal{E}$. Since the shears are the logarithms of the cross-ratios and the cross-ratios are bounded away from $0$ and $\infty$, the limits of the difference quotients uniformly converge to the derivatives of the shears. Therefore, the complexified map is holomorphic.
\end{proof}

\subsection{The geometry of $\mathbb{H}^3$}
A model of hyperbolic 3-space we will be considering is the upper half space $\mathbb{H}^3=\{(x,y,z)\in \mathbb{R}^3: z>0\}$. 
The geodesic subspaces of co-dimension $1$ are either Euclidean  half-planes orthogonal to the ideal boundary  $\partial \mathbb{H}^3=\hat{\mathbb{C}}$  along a line or a Euclidean hemispheres whose boundary is a circle in $\hat{\mathbb{C}}$. In either case, these are isometric embeddings of $\mathbb{H}^2$ in $\mathbb{H}^3$ for their hyperbolic metrics. The geodesic subspaces of co-dimension $1$, called hyperbolic planes, divide $\mathbb{H}^3$ into two components which we call half-spaces. A geodesic in $\mathbb{H}^3$ is either a Euclidean ray orthogonal to $\mathbb{C}$ or a Euclidean semi-circle orthogonal to $\mathbb{C}$. 

The orientation preserving isometries of $\mathbb{H}^3$ restrict to the M{\"o}bius transformations on the boundary $\hat{\mathbb{C}}$, i.e.,  $PSL_2(\mathbb{C})$.  Since the actions of $PSL_2(\mathbb{C})$ can be described as a sequence of reflections over circles, the corresponding isometry in $\mathbb{H}^3$ is realized through sequences of reflections of half-planes and hemispheres whose boundary circles are the ones where the reflection takes place. 
An isometry of $\mathbb{H}^3$ is called {\it elliptic} if it pointwise fixes a unique geodesic and rotates all other points around the fixed geodesic by a fixed angle. 
An isometry is {\it parabolic} if it fixes a single point on $\hat{\mathbb{C}}$. 
An isometry is {\it hyperbolic} if it has two fixed points on $\hat{\mathbb{C}}$, translates points on the geodesic connecting the two points by a fixed amount, and set-wise fixes any plane containing this geodesic. This geodesic is called the {\it translation axis}. A hyperbolic isometry is obtained by the composition of two hyperbolic reflections in hyperbolic planes that are disjoint and their ideal boundaries are also disjoint. The translation axis is the unique orthogonal geodesic to the both hyperbolic planes. 
An isometry is {\it  loxodromic} if it is a composition of a hyperbolic isometry followed by an elliptic isometry whose set of fixed points lie on the translation axis of the hyperbolic isometry. This geodesic is also called the translation axis of the loxodromic element.

For $A\in PSL_2(\mathbb{C})$ a loxodromic element, the translation length is given by $\lambda_A=arccosh(\frac{1}{2}tr(A))$. When the trace is real,  $A$ is hyperbolic and $\lambda_A$ is real.  

Let $T^1\mathbb{H}^3$ be the unit tangent bundle of $\mathbb{H}^3$. Elements of $T^1\mathbb{H}^3$ are denoted by $(x,v)$ for $x\in \mathbb{H}^3$ and $v\in T_x\mathbb{H}^3$ the tangent space to $\mathbb{H}^3$ at point $x$ where $||v||_{T_x(\mathbb{H}^3)}=1$. The metric on $T^1\mathbb{H}^3$ can be defined by the following: if $(x,v),(y,u)\in T^1\mathbb{H}^3$, and $\gamma$ is a 
geodesic between $x$ and $y$ in $\mathbb{H}^3$, let $\tau:T_x\mathbb{H}^3\rightarrow T_y\mathbb{H}^3$ be the parallel transport  from $x$ to $y$  along $\gamma$. For $\Theta[(x,v),(y,u)]=||\tau (x,v)-(y,u)||$,  $$d[(x,v),(y,u)]^2=\Theta[(x,v),(y,u)]^2+\text{length}(\gamma)^2$$

We will need the following  lemmas in the next section.
\begin{lem}
\label{lem:vect-norm}
Let $K\subset \mathbb{H}^3$ be a fixed compact set, and let $(w,v)\in T^1\mathbb{H}^3$ with $w\in K$. Then there exists $C=C(K)>0$ such that 
$$
d_{T^1\mathbb{H}^3}((w,v),A(w,v))\leq C\| A-I\| ,
$$
where $A\in PSL_2(\mathbb{C})$ has translation axis intersecting $K$ and purely imaginary translation length.
\end{lem}

\begin{proof}
The action of $A$ on the upper half-plane is given by a formula that depends on the entries of the matrix representation of $A$ (see \cite{Beardon}). The derivatives are bounded in terms of these entries, and the image of a unit tangent vector is a linear combination of the derivatives. These are bounded by a constant multiple of $\| A-I\|$. For more details, see \cite[Lemma II.3.3.2]{EpsteinMarden} or \cite[Lemma 5.1]{Saric14}.
\end{proof}

The following lemma is established in \cite{KeenSeries} (see also \cite{Saric,Saric13,Saric14}).

\begin{lem}
\label{lem:inclusion}
Let $(p_n,v_n),(p_{n+1},v_{n+1})\in T^1\mathbb{H}^3$ with $p_n,p_{n+1}$ two consecutive points of the division on the geodesic ray such that the unit tangent vectors $v_n,v_{n+1}$ are  pointing in the direction of the ray. Then there exists $\delta >0$ such that if 
$d_{T^1\mathbb{H}^3}((p_{n+1},v_{n+1}),(q,w))<\delta$ then
$$
C(p_n)\supset\overline{C(q)},
$$
where $C(q)$ is the half-space with boundary geodesic plane passing through $q$ and orthogonal to $w$, and $w$ pointing inside $C(q)$.
\end{lem}

The following lemma is first stated in \cite{Bonahon} (see \cite[Lemma 5.5 ]{Saric14} for a proof).

\begin{lem}
\label{lem:diff-est-opposite}
Let $a$ be a compact geodesic arc in $\mathbb{H}^3$. Then there exists a constant $C>0$ such that for any $\epsilon >0$ and any two geodesics $g$ and $h$ of $\mathbb{H}^3$ that intersect $a$ and have a common ideal endpoint, we have
\begin{equation}
\label{lem:cancellations}
\| T_g^{\epsilon}\circ T_h^{-\epsilon}-I\| \leq C\epsilon |a|_{g,h} ,
\end{equation}
where $|a|_{g,h}$ is the hyperbolic length of the part of $a$ between $g$ and $h$.
\end{lem}

\subsection{The map from complex shear coordinates to $QF(X)$} 

Let $s_0=s_{[f^{\mu_0}]}:\mathcal{E}\to\mathbb{R}$ be the shear coordinate parameters of a point in $T(X)$. The zero-shear coordinate (${\bf 0}(e):=0$ for all $e\in\mathcal{E}$) corresponds to the basepoint $[id]\in T(X)$. Given $\epsilon_0>0$, define the $\epsilon_0$-neighborhood of the zero shear in $s_{*}(T(X))$ by 
$$U_{\epsilon_0} ({\bf 0})=\{  s_{f^{[\mu ]}}:[f^{\mu}]\in T(X)\ \mathrm{and}\ \|  s_{f^{[\mu ]}}\|_{\infty}<\epsilon_0\}. $$ 
We define a complex neighborhood of the zero-shear coordinate by
$$
V_{\epsilon_0}({\bf 0})=\{ \tau\cdot s:s\in U_{\epsilon_0}({\bf 0})\ \mathrm{and}\ \tau\in\mathbb{C}\ \mathrm{with}\ |\tau |<1\} .
$$

In other words, $V_{\epsilon_0}({\bf 0})$ consists of all functions $s\in \ell^{\infty}(\mathcal{E},\mathbb{C})$ that are invariant under the action of the covering group $\Gamma$ of $X$, the sum of the values of $s$ on the geodesics at each puncture is zero and $\| s\|_{\infty}<\epsilon_0$.

The complex neighborhood of $s_0:=s_{[f^{\mu_0}]}$ is defined by
$$
V_{\epsilon_0}({s_0})=s_0+V_{\epsilon_0}({\bf 0}).
$$
The shear coordinates $s_{[f^{\mu}])}\in \ell^{\infty}(\mathcal{E},\mathbb{R})$ are invariant under the action of the covering group $\Gamma$ of $X$ on the set of edges $\mathcal{E}$. Each fan of edges in $\mathcal{E}$ (i.e., the set 
of edges with a common endpoint) corresponds to the lifts of edges of the triangulation $\mathcal{T}$ of $X$ meeting at a puncture of $X$. This implies that to each vertex of $\mathcal{T}$, there corresponds a parabolic subgroup of the covering group $\Gamma$ of $X$,  whose set of fixed points is the vertex of the fan.  Given a vertex $p$, an edge $e_1$ with vertex $p$ and the primitive parabolic element $A$ of the covering group of $X$ with fixed point $p$, let $(e_1,e_2,\ldots ,e_n)$ be the set of adjacent edges with vertex $p$  between $e_1$ and $A(e_1)$ (not including $A(e_1)$). For each $s_{[f^{\mu}]}$ with $[f^{\mu}]\in T(X)$, we have 
$$
\sum_{j=1}^ns_{[f^{\mu}]}(e_j)=0.
$$
Moreover, there is an upper bound on all $n$ over all vertices of $\mathcal{F}$.

We prove that there is an extension of the map $s_{[f^{\mu}]}+U_{\epsilon_0}({\bf 0})\to T(X)$ from the real shear coordinates to the Teichm\"uller to a map from the complex shear coordinates of $s_{[f^{\mu}]}+V_{\epsilon_0}({\bf 0})$ to the QuasiFuchsian space $QF(X)$, and that the extension is holomorphic.

\begin{thm}
\label{thm:complex_shears-qf}
Let $[f^{\mu_0}]\in T(X)$ and $s_0=s_{[f^{\mu_0}]}\in \ell^{\infty}(\mathcal{E},\mathbb{R})$ be its shear coordinates. Then there exists $\epsilon_0=\epsilon_0([f^{\mu_0}])>0$ such that the inverse of the shear map extends to a holomorphic map
$$
s_{*}^{-1}:V_{\epsilon_0}({s_0})\to QF(X)
$$
of a complex neighborhood $V_{\epsilon_0}({s_0})=s_0+V_{\epsilon_0}({\bf 0})\subset QF(X)$. 
\end{thm}

\begin{proof}
By \cite[Theorem 14.9, page 198]{Chae}, it is enough to prove that $s_*^{-1}$ is bounded on $V_{\epsilon_0}({s_0})$ and that $\zeta \mapsto (s_*)^{-1}(s_0+s_1'+\zeta s_1)$ is a holomorphic function of the variable $\zeta$ into $T(X)$ for all $\zeta\in\mathbb{C}$ with $\| s_1'+\zeta s_1\|_{\infty}<\epsilon_0$ and $s_1',s_1\in V_{\epsilon_0}({\bf 0})$. 

Given $s\in V_{\epsilon_0}({s_0})$, we first define the map on the vertices $\hat{\mathbb{Q}}=\mathbb{Q}\cup\{\infty\}$ of the Farey triangulation whose shears are  $s$. We start by normalizing the map to be the identity on the triangle $\Delta_0$ of $\mathcal{F}$ with vertices $0$, $1$, and $\infty$. Let $\Delta$ be another triangle of $\mathcal{F}$ and $l$ an oriented geodesic arc connecting $\Delta_0$ to $\Delta$. Let $\{ g_0,g_1,\ldots ,g_n\}$ be the set of all edges in $\mathcal{E}$ which separate $\Delta_0$ from $\Delta$ such that the point $g_i\cap l$ is before the point $g_{i+1}\cap l$ for the orientation of $l$ and for all $i=0,\ldots ,n-1$. We orient each $g_i$ such that $\Delta_0$ is to its left.

We write $s=s_0+ s_1'+\zeta s_1$, where $s_1,s_1'\in V_{\epsilon_0} ({\bf 0})$ and $\zeta\in\mathbb{C}$ with $\| s_1'+\zeta s_1\|_{\infty}<\epsilon_0$. On the vertices of the triangle $\Delta$, the map is given by the developing (or cocycle) map according to the complex shears $s$ (see \cite{EpsteinMarden}). Namely, the map is
\begin{equation}
\label{eq:cocycle}
T^{s(g_0)}_{g_0}\circ T^{s(g_1)}_{g_1}\circ\cdots\circ T^{s(g_n)}_{g_n} (z),
\end{equation}
where $T_{g_i}^{s(g_i)}(z)$ is as a loxodromic element of $PSL_2(\mathbb{C})$ with fixed points $r_i$ (the initial point of $g_i$) and $a_i$ (the endpoint of $g_i$), and the complex translation length $s(g_i)=s_0(g_i)+s_1'(g_i)+\zeta s_1(g_i)$.  
Note that
$$
\frac{T_{g_i}^{s(g_i)}(z)-a_i}{T_{g_i}^{s(g_i)}(z)-r_i}=e^{s_0(g_i)+s_1'(g_i)+\zeta s_1(g_i)}\frac{z-a_i}{z-r_i}
$$
which gives
$$
T_{g_i}^{s(g_i)}(z)=\frac{(r_i-a_ie^{s_0(g_i)+s_1'(g_i)+\zeta s_1(g_i)})z+a_ir_i(e^{s_0(g_i)+s_1'(g_i)+\zeta s_1(g_i)}-1)}{(1-e^{s_0(g_i)+s_1'(g_i)+\zeta s_1(g_i)})z+r_ie^{s_0(g_i)+s_1'(g_i)+\zeta s_1(g_i)}-a_i}.
$$
In particular, the composition $T^{s(g_0)}_{g_0}\circ T^{s(g_1)}_{g_1}\circ\cdots\circ T^{s(g_n)}_{g_n} (z)$ is holomorphic in $\zeta$ for a fixed $z\in\bar{\mathbb{C}}$ as long as $T^{s(g_0)}_{g_0}\circ T^{s(g_1)}_{g_1}\circ\cdots\circ T^{s(g_n)}_{g_n} (z)\neq \infty$. We will prove below that for all $z\in\mathbb{Q}$, $T^{s(g_0)}_{g_0}\circ T^{s(g_1)}_{g_1}\circ\cdots\circ T^{s(g_n)}_{g_n} (z)\neq \infty$ for $\epsilon_0>0$ small enough and for the corresponding $\zeta$ with $\| s_1'+\zeta s_1\|_{\infty}<\epsilon_0$ (where $s=s_0+s'_1+\zeta s_1$)  which will establish that the map is well-defined on $\hat{\mathbb{Q}}$ and holomorphic.

Our first observation is that 
\begin{equation}
\label{eq:conjugation}
(A\circ T_{g_i}^{s(g_i)}\circ A^{-1})(z)=T_{A(g_i)}^{ s(g_i)}(z),
\end{equation} 
where $A\in PSL_2(\mathbb{C})$ and $A(g_i)$ is the geodesic in $\mathbb{H}^3$ whose endpoints are $A(r_i)$ and $A(a_i)$. In other words, $T_{A(g_i)}^{ s(g_i)}(z)$ is the M\"obius map with fixed points $A(r_i)$ and $A(a_i)$, and the translation length $s(g_i)$. 

At this point we represent the shear function $s\in U_{\epsilon_0}(s_0)$ by $$s=s_0+\zeta s_1,$$ where $s_0,s_1\in s_*(T(X))$, $\| s_1\|_{\infty}=\epsilon_0$, and $\zeta=\xi +i\eta\in\mathbb{C}$ with $|\zeta |<1$. Then we have $s(g_i)=s_0(g_i)+\zeta s_1(g_i)=[s_0(g_i)+\xi s_1(g_i)]+i \eta s_1(g_i)$, where $s_0(g_i)+\xi s_1(g_i)$ and $\eta s_1(g_i)$ are real-valued. By (\ref{eq:conjugation}), the developing map in (\ref{eq:cocycle}) can be written as
\begin{equation}
\label{eq:real-imaginary}
\Big{[}T^{i\eta  s_1(g_0)}_{g_0'}\circ \cdots\circ T^{i\eta   s_1(g_n)}_{g_n'}\Big{]}\circ 
\Big{[}T^{ s_0(g_0)+\xi s_1(g_0)}_{g_0}\circ \cdots\circ T^{s_0(g_n)+\xi s_1(g_n)}_{g_n}\Big{]},
\end{equation}
where $g_i'=\Big{[}T^{ s_0(g_0)+\xi s_1(g_0)}_{g_0}\circ \cdots\circ T^{s_0(g_i)+\xi s_1(g_i)}_{g_i}\Big{]}(g_i)$. The formula (\ref{eq:real-imaginary}) decomposes the developing map from (\ref{eq:cocycle}) into the (real) shear map that preserves the real line and the pure (imaginary) bending part (see \cite{EpsteinMarden}). The real part of the developing map moves the support geodesics of the imaginary part.

By Theorem \ref{thm:Teich-param}, the (real) developing map corresponding to the real shears $s_0+\xi s_1$ is a quasisymmetric map since $s_0+\xi s_1\in \ell^{\infty}(\mathcal{E},\mathbb{R})$. We denote by $\mathcal{F}'$ the image of $\mathcal{F}$ under the real developing map. It remains to map with the purely imaginary developing map 
 $$
 \mathcal{B}=T^{i\eta   s_1(g_0)}_{g_0'}\circ \cdots\circ T^{i\eta  s_1(g_n)}_{g_n'}
 $$ 
 given by the purely imaginary shear coordinates starting from $\mathcal{F}'$.
  
We prove that $\mathcal{B}$ is injective on the vertices of $\mathcal{F}'$.
Indeed, let $x,y\in\mathbb{R}\subset \hat{\mathbb{C}}=\partial\mathbb{H}^3$ be two different points that are vertices of $\mathcal{F}'$. Let $h$ be a geodesic connecting $x$ and $y$ in $\mathbb{H}^2\subset \mathbb{H}^3$. If $h$ is an edge in $\mathcal{F}'$ then $\mathcal{B}|_h$ is an isometry of $\mathbb{H}^3$ and therefore $\mathcal{B}(x)\neq\mathcal{B}(y)$. 

The main case is when $h$ transversely intersects edges of $\mathcal{F}'$. Let $\Delta_0'$ be a triangle of $\mathcal{F}'$ that intersects $h$ and fixes a point $p\in h\cap\Delta_0'$. The point $p$ divides $h$ into two infinite rays, and we orient them to start at $p$.
Divide each ray of $h$ into arcs of length $1$ starting from point $p$. Denote the points of the division of $h$ by $\{ p_n\}_{n\in\mathbb{Z}}$ with $p_0=p$. The points $p_n$ and $p_{n+1}$ for $n\geq 0$ are adjacent and lie on one ray with indices increasing along the ray. The points $p_{n+1}$ and $p_n$ for $n\leq -1$ are adjacent and lie on the other ray with indices decreasing along the ray. 

For $n\neq 0$, let $C(p_n)$ be the union of the open half-space in $\mathbb{H}^3$ whose boundary is a geodesic plane orthogonal to $h$ passing through $p_n$ that contains the infinite part of the ray to which $p_n$ belongs to and the open disk on the boundary $\bar{\mathbb{C}}$ of the open half-space. Let $C^+(p_0)$ be the open half-space (together with the open disk on its boundary on $\bar{\mathbb{C}}$) whose boundary in $\mathbb{H}^3$ is the geodesic plane through $p_0$ orthogonal to $h$ containing all $p_n$ with $n>0$. Let $C^-(p_0)$ be the open half-space (together with the open disk on its boundary on $\bar{\mathbb{C}}$) whose boundary in $\mathbb{H}^3$ is the geodesic plane through $p_0$ orthogonal to $h$ containing all $p_n$ with $n<0$. 
Denote by $\overline{C(p_n)}$ the Euclidean closure of $C(p_n)$ in the model $\mathbb{H}^3\cup\bar{\mathbb{C}}$.

Then we have
$$
C^+(p_0)\supset \overline{C(p_1)},\ \ C(p_n)\supset \overline{C(p_{n+1})}
$$
for all $n\geq 1$, 
and
$$
C^-(p_0)\supset \overline{C(p_{-1})},\ \  C(p_n)\supset \overline{C(p_{n-1})}
$$
for all $n\leq -1$. 
We will prove that the above inclusions remain in force after applying the developing map $\mathcal{B}$ when $\epsilon_0>0$ is small enough.  This implies $\mathcal{B}(x)\neq \mathcal{B}(y)$ since the (open) boundary disks of $C^+(p_0)$ and $C^-(p_0)$ are disjoint, and the nesting property guarantees that $\mathcal{B}(x)$ is in the boundary disk of $C^+(p_0)$ and $\mathcal{B}(y)$ is in the boundary disk of $C^-(p_0)$. 

It is enough to prove that $\mathcal{B}(C(p_n))\supset\overline{\mathcal{B}(C(p_{n+1}))}$ for all $n>0$ and the statement for $n<0$ and $n=0$ follows by the same reasoning.
By pre-composing and post-composing the developing map $\mathcal{B}$ with hyperbolic isometries, we can assume that $p_n=(0,0,1)\in\mathbb{H}^3$, $h$ is geodesic orthogonal to the $z$-axis with endpoints $(1,0,0)$ and $(-1,0,0)$ on the $x$-axis, $p_{n+1}$ has positive $x$-coordinate and $\mathcal{B}$ is normalized such that it equals to the identity on the triangle of $\mathcal{F}'$ that contains $p_n$. Then $C(p_n)=\mathcal{B}(C(p_n))$ is the half-space of $\mathbb{H}^3$ whose boundary is the geodesic plane which is the Euclidean half-plane above the $y$-axis and whose points have positive $x$-coordinates. 

The restriction of $\mathcal{B}$ to the triangle that contains $p_{n+1}$ is a M\"obius map $B_{n+1}\in PSL_2(\mathbb{C})$. Let $v_{n+1}$ be the unit tangent vector to $h$ at $p_{n+1}$ in the direction of the ray containing $p_{n+1}$. 
By Lemma \ref{lem:vect-norm}, the distance from $(p_{n+1},v_{n+1})$ to $(\mathcal{B}(p_{n+1}),d\mathcal{B}(v_{n+1}))$ as elements of the unit tangent bundle $T^1\mathbb{H}^3$ is bounded above by $D\cdot\| B_{n+1}-I\|$, where $I$ is the identity and $D>0$ is a constant. By Lemma \ref{lem:inclusion}, there is $\delta >0$ such that $\overline{C(q_{n+1},w_{n+1})}\subset C(p_n,v_n)$ if $d_{T^1\mathbb{H}^3}((q_{n+1},w_{n+1}),(p_{n+1},v_{n+1}))<\delta$. 

Therefore, it remains to prove that $\| B_{n+1}-I\|$ can be made arbitrarily small for $\epsilon_0>0$ small enough. The bound on $\| B_{n+1}-I\|$ follows by a geometric property of cancellations in the long compositions in the definition of the developing map for shears with bounded variations. This general idea goes back to papers \cite{Bonahon,Saric13,Saric14}, but the geometric situation and the estimates are new.

Depending on the relative position, the part of the geodesic between $p_n$ and $p_{n+1}$ can intersect many edges of $\mathcal{F}'$. There is no upper bound on the number of edges it can intersect, even though its length is $1$. This phenomenon happens because the distance between edges in the same fan of $\mathcal{F}'$ is zero. Let $\{ g_1',\ldots ,g_j'\}$ be the set of consecutive edges of $\mathcal{F}'$ that the geodesic arc between $p_n$ and $p_{n+1}$ intersects. We group consecutive edges into fans as follows. Namely, $\{ g_1',\ldots ,g_{j_1}'\}$ is the first group which belongs to fan with vertex $q_1$, then $\{ g_{j_1}',\ldots ,g_{j_2}'\}$ is the second group which belongs to fan with vertex $q_2$ with $q_2\neq q_1$, and so on. Note that two consecutive groups share an edge; the first two groups share $g_{j_1}'$. If we have three consecutive groups of geodesics belonging to different fans, then the distance between the last geodesic in the first group and the first geodesic in the third group is bounded below by a positive constant. To see this, note that the distance for the Farey triangulation is at least the distance between two opposite sides of an ideal rectangle obtained by the union of two adjacent triangles. Since the shears for $\mathcal{F}$ are all zero, this distance is the same for any two edges in $\mathcal{F}$. Since $\mathcal{F}'$ is the image of $\mathcal{F}$ under a quasisymmetric map, this distance in $\mathcal{F}'$ is bounded below by a positive constant (see \cite[page 81]{LV}, \cite{GardinerLakic}). 

Since there is a lower bound on the distance between edges in $\{ g_1',\ldots ,g_j'\}$ that do not belong to adjacent fans and the distance between $p_n$ and $p_{n+1}$ is $1$, it follows that the geodesic arc $[p_n,p_{n+1}]$ between $p_n$ and $p_{n+1}$ can intersect at most a fixed number $k_0$ of different fans. The number $k_0$ is uniform over all pairs of vertices $x,y$ on $\mathcal{F}'$ and over all geodesic arcs $[p_n,p_{n+1}]$, while the total number $j$ of geodesics that $[p_n,p_{n+1}]$ intersects goes to infinity. 

The developing map $\mathcal{B}$ is a composition of at most $k_0$ developing maps corresponding to non-zero shears in a single fan. We estimate the norm of the developing map for each fan separately. We consider the developing map
$$
T^{i\eta   s_1(g_1)}_{g_1'}\circ \cdots\circ T^{i\eta  s_1(g_{j_1})}_{g_{j_1}'}
$$
corresponding to the adjacent edges of the fan with vertex $q_1$. Denote by $\langle A_{q_1}\rangle$ the parabolic M\"obius subgroup  of the Fuchsian covering group $\Gamma_X$ of $X$ that fixes $q_1$. Then $\langle A_{q_1}\rangle$ identifies the edges of the fan associated with $q_1$. We recall that the number of different orbits under $\langle A_{q_1}\rangle $ of the edges in the fan associated with $q_1$ is finite (since the triangulation of $X$ has an upper bound of edges at each puncture), and there is an upper bound on the number of different orbits for all fans.

We partition the set $\{ g_1',\ldots ,g_{j_1}'\}$ into groups of adjacent edges such that the image under the primitive parabolic element $A_{q_1}$ maps the first element of one group to the first element of the following group plus the tail group which has no following group. In other words, there are $m,j,l>0$ and $e\geq 0$ such that $m$ groups of adjacent edges $\{ g_{k,1}',\ldots ,g_{k,j}'\}_{k=1}^m$ have property $A_{q_1}(g_{k,1}')=A_{q_1}(g_{k+1,1}')$ for $k=1,\ldots ,m-1$; followed by $\{ g_{m+1,1}',\ldots ,g_{m+1,e}'\}$ with $e< j$ and $A_{q_1}(g_{m,1}')=A_{q_1}(g_{m+1,1}')$ if $e>0$. We also recall that $j$ has an upper bound when considering all fans.

We first find an estimate for 
\begin{equation}
\label{eq:develop-fan}
T^{i\eta  s_1(g_{k,1})}_{g_{k,1}'}\circ T^{i\eta   s_1(g_{k,2})}_{g_{k,2}'} \circ T^{i\eta  s_1(g_{k,3})}_{g_{k,3}'}\circ \cdots\circ T^{i\eta   s_1(g_{k,j-1})}_{g_{k,j-1}'}\circ T^{i\eta  s_1(g_{k,j})}_{g_{k,j}'} .
\end{equation}
Note that 
\begin{equation}
\label{eq:sum-shears}
s_1(g_{k,1})+\cdots +s_1(g_{k,j})=0,
\end{equation} 
the geodesics $g_{k,1}',\ldots ,g_{k,j}'$ have a common endpoint $q_k$ and $|\eta |<1$ by the assumption. 

Lemma \ref{lem:diff-est-opposite} states that given a closed geodesic arc $a$ in $\mathbb{H}^3$, there exists a constant $C>0$ such that for any $\epsilon >0$ and any two geodesics $g$ and $h$ of $\mathbb{H}^3$ that intersect $a$ and have a common endpoint, we have
\begin{equation}
\label{eq:cancellations}
\| T_g^{\epsilon}\circ T_h^{-\epsilon}-I\| \leq C |a|\epsilon ,
\end{equation}
where $|a|$ is the hyperbolic length of $a$.

We rewrite the composition (\ref{eq:develop-fan}) by adding extra terms such that it becomes a composition of pieces of the form in (\ref{eq:cancellations}). Namely, we replace $T^{i\eta   s_1(g_{k,1})}_{g_{k,1}'}\circ T^{i\eta  s_1(g_{k,2})}_{g_{k,2}'}$ with
$$
T^{i\eta  s_1(g_{k,1})}_{g_{k,1}'}\circ T^{-i\eta  s_1(g_{k,1})}_{g_{k,2}'}\circ  T^{i\eta  [s_1(g_{k,1})+s_1(g_{k,2})]}_{g_{k,2}'} .
$$
Then we replace $T^{i\eta  [s_1(g_{k,1})+s_1(g_{k,2})]}_{g_{k,2}'}\circ T^{i\eta  s_1(g_{k,3})}_{g_{k,3}'}$ with
$$
T^{i\eta  [s_1(g_{k,1})+s_1(g_{k,2})]}_{g_{k,2}'}\circ T^{-i\eta  [s_1(g_{k,1})+s_1(g_{k,2})]}_{g_{k,3}'}\circ T^{i\eta  [s_1(g_{k,1})+s_1(g_{k,2})+s_1(g_{k,3})]}_{g_{k,3}'} .
$$
Continuing in this fashion, in the last step, we obtain
$$
T^{i\eta [s_1(g_{k,1})+s_1(g_{k,2})+\cdots +s_1(g_{k,j-1})]}_{g_{k,j-1}'}\circ T^{i\eta  s_1(g_{k,j})}_{g_{k,j}'} 
$$
which by (\ref{eq:sum-shears}) equals to
$$
T^{-i\eta s_1(g_{k,j})}_{g_{k,j-1}'}\circ T^{i\eta  s_1(g_{k,j})}_{g_{k,j}'} .
$$
Therefore, the composition in (\ref{eq:develop-fan}) is rewritten as a composition of $j$ pairs of rotations by angles with opposite signs. Denote by 
\begin{equation}
\label{eq:pairs} 
\begin{split}B_1=T^{i\eta  s_1(g_{k,1})}_{g_{k,1}'}\circ T^{-i\eta  s_1(g_{k,1})}_{g_{k,2}'},B_2=T^{i\eta  [s_1(g_{k,1})+s_1(g_{k,2})]}_{g_{k,2}'}\circ T^{-i\eta  [s_1(g_{k,1})+s_1(g_{k,2})]}_{g_{k,3}'},\\ \ldots ,B_j=T^{-i\eta s_1(g_{k,j})}_{g_{k,j-1}'}\circ T^{i\eta  s_1(g_{k,j})}_{g_{k,j}'} 
\end{split}
\end{equation} 
the composition of these pairs of rotations by angles with opposite signs. The maximum of the absolute value of the angles is at most $ (j\| s_1\|_{\infty})|\eta |$. Note that $j\| s_1\|_{\infty}$ is a fixed constant for the whole triangulation and the neighborhood $V_{\epsilon_0}({s_0})$.

By (\ref{eq:cancellations}), for all $u=1,2,\ldots ,j$, we obtain
\begin{equation}
\label{eq:est-B}
\| B_u-I\|\leq C d_u  |\eta |,
\end{equation}
where $C$ is a universal constant and $d_u$ is the length of the piece of the geodesic arc $[p_n,p_{n+1}]$ between $g_{k,u-1}'$ and $g_{k,u}'$.

By (\ref{eq:est-B}), we obtain
$$
\log \| B_1\circ B_2\circ\cdots\circ B_j\|\leq\sum_{u=1}^j \log\|B_u\|\leq\sum_{u=1}^j\log (1+C d_u |\eta |)\leq C(\sum_{u=1}^j d_u)|\eta |.
$$
Recall that we partitioned edges $\{ g_1,\ldots ,g_{j_1}\}$ into $m$ groups and each group satisfies the above estimate. The number of groups $m$ is not bounded. However, by adding over all groups, we obtain that the norm of the composition is bounded since the above estimate is a multiple of the length of the part of $[p_n,p_{n+1}]$ between the edges. The length of $[p_n,p_{n+1}]$ is $1$ which gives the uniform bound. The last group $\{ g_{m+1,1},\ldots ,g_{m+1,e}\}$ with $e< j$ has finitely many elements, and the composition for the developing map of this group is bounded by a constant because each element is a rotation.
Moreover, the geodesic arc $[p_n,p_{n+1}]$ intersects at most finitely many fans, and again, the norm of the composition is bounded by a constant $C$ (uniformly for all geodesics and all arcs $[p_n,p_{n+1}]$).

We prove that the composition in (\ref{eq:develop-fan}) is close to the identity. Indeed, note that
\begin{equation}
\label{eq:composition-group}
\begin{split}
\| B_1\circ B_2\circ\cdots \circ B_j-I\|\leq \| B_1\circ B_2\circ\cdots \circ B_j-B_1\circ B_2\circ\cdots \circ B_{j-1}\| +\\
\| B_1\circ B_2\circ\cdots \circ B_{j-1}-B_1\circ B_2\circ\cdots \circ B_{j-2}\|+\cdots +\| B_1-I\|.
\end{split}
\end{equation}
By (\ref{eq:est-B}) and the above bound, we have
\begin{equation}
\label{eq:composition-group-inter}
\begin{split}
\| B_1\circ B_2\circ\cdots \circ B_{h}-B_1\circ B_2\circ\cdots \circ B_{h-1}\|\leq \\ \| B_1\circ B_2\circ\cdots \circ B_{h-1}\|\cdot\| B_{h}-I\|\leq C d_h|\eta |.
\end{split}
\end{equation}
By (\ref{eq:composition-group}) and (\ref{eq:composition-group-inter}), we have
\begin{equation}
\label{eq:composition-group-main}
\| B_1\circ B_2\circ\cdots \circ B_j-I\|\leq C|\eta |.
\end{equation}

We are ready to prove that the developing map is close to the identity when $\epsilon_0>0$ is small enough. Let $\Delta_{p_n}$ and $\Delta_{p_{n+1}}$ be the triangles that contain $p_n$ and $p_{n+1}$, respectively. Then $\mathcal{B}|_{\Delta_{p_n}}=I$ and $\mathcal{B}|_{\Delta_{p_{n+1}}}$ is the composition of the finitely many groups (at most $k_0$) rotations (whose support geodesics belong to a single fan) as above.

Using the triangle inequality, (\ref{eq:composition-group-main}) and the fact that the composition of rotations for each group has uniformly bounded  norm, it follows that
\begin{equation}
\label{eq:est-develop-main}
\| \mathcal{B}|_{\Delta_{p_{n+1}}}-I\|\leq Ck_0|\eta |.
\end{equation}
Since $k_0$ is an upper bound on the number of groups of geodesics over all geodesic arcs $[p_n,p_{n+1}]$ and over all $x,y\in\bar{\mathbb{R}}\subset\bar{\mathbb{C}}$, and $|\eta |<\epsilon_0$, it follows that $\| \mathcal{B}|_{\Delta_{p_{n+1}}}-I\|$ can be made small by choosing $\epsilon_0>0$ small enough. Therefore
$$
\mathcal{B}(x)\neq\mathcal{B}(y)
$$
when $x\neq y$. 

Since $\mathcal{B}(\infty )=\infty$, it follows that $T^{s(g_0)}_{g_0}\circ T^{s(g_1)}_{g_1}\circ\cdots\circ T^{s(g_n)}_{g_n} (z)=\mathcal{B}(z)\neq\infty$ for $z\in\mathbb{R}$. Therefore $\zeta\mapsto \mathcal{B}(z)$ is holomorphic in $\zeta$ for each $z\in\mathbb{R}$. Let $\mathcal{B}=\mathcal{B}''\circ\mathcal{B}'$, where $\mathcal{B}'$ is the developing map for the complex shears $s_0+s_1'$. By the above, the developing map $\mathcal{B'}$ is injective on $\mathbb{Q}$ and the image $\mathcal{B}'(\mathbb{Q})$ is a subset of $\mathbb{C}$. The map $\mathcal{B}''$ depends on the complex parameter $\zeta$ and we established that for each $z\in\mathbb{Q}$ the map $\zeta\to\mathcal{B}(z)$ is holomorphic in $\zeta$. We note that (\ref{eq:conjugation}) implies that $\zeta\to\mathcal{B}''(z)$ is also holomorphic.

The map $\mathcal{B}''$ is injective on $\mathcal{B}'(\mathbb{Q})$ because $\mathcal{B}$ is injective on $\mathbb{Q}$. Moreover, the map $\mathcal{B}''$ is the identity on $\mathcal{B}'(\mathbb{Q})$ when $\zeta =0$. Therefore, the map $\mathcal{B}''$ is a holomorphic motion of $\mathcal{B}'(\mathbb{Q})$ in the parameter $\zeta$. It extends to the holomorphic motion of the closure of $\mathcal{B}'(\mathbb{Q})$ by the lambda lemma of Man\`e-Sad-Sullivan \cite{MSS}. By Slodkowski's theorem \cite{Slodkowski},  $\mathcal{B}''$ can be extended to a holomorphic motion of  $\mathbb{C}$ that conjugates Fuchsian group $\Gamma_X$ to other Fuchsian groups for each parameter $\zeta$ (see \cite{EKK}). Therefore, the map
$(\zeta ,z)\mapsto \mathcal{B}(z)$ is a holomorphic motion of $\mathbb{C}$ and the induced map that send $\zeta$ to the Beltrami coefficient of the extension of $\mathcal{B}|_{\mathbb{R}}$ is holomorphic map into $L^{\infty}(\mathbb{C})$. The projection to $T(X)$ is also holomorphic and we proved that the map $\zeta \mapsto (s_*)^{-1}(s_0+s_1'+\zeta s_1)$ is holomorphic.
\end{proof}

\section{The mapping class groups and triangulations}

Let $X$ be a Riemann surface with a bounded ideal hyperbolic triangulation $\mathcal{T}$. 
We consider the action of an element $[f]$ of the big mapping class group $MCG(X)$ on the triangulation $\mathcal{T}$ of $X$. 

Since $X$ has a bounded ideal hyperbolic triangulation, it is quasiconformal to a Riemann surface with an ideal triangulation whose all shears are zero (see Theorem \ref{thm:Teich-param}). The lift to $\mathbb{H}^2$
 of a triangulation of an infinite surface with all zero shears does not accumulate in $\mathbb{H}^2$. Therefore, the Fuchsian covering group of $X$ is of the first kind. By \cite[Theorem 3.6]{Saric23}, a homeomorphism $f:X\to X$ lifts to a homeomorphism of $\mathbb{H}^2$ that extends to a homeomorphism of the extended real line $\hat{\mathbb{R}}$.  We prove
 
 \begin{thm}
 \label{thm:qmcg-triangulations}
 Let $X$ be a Riemann surface with a bounded ideal hyperbolic triangulation. Then a homeomorphism $f:X\to X$ is homotopic to a quasiconformal map if and only if the triangulations $\mathcal{T}$ and $\mathcal{T}_{[f]}$ have bounded intersections, where $\mathcal{T}_{[f]}$ is found by replacing each edge of the image triangulation $f(\mathcal{T})$ with an ideal hyperbolic geodesic homotopic to it modulo its ideal endpoints that are at the punctures of $X$.
 \end{thm}
 
 \begin{proof}
 We lift $f:X\to X$ to the universal covering $\tilde{f}:\mathbb{H}^2 \to\mathbb{H}^2$. The lift extends $\tilde{f}$ extends to a homemorphism of $\hat{\mathbb{R}}$ that setwise preserves the endpoints of the lift $\tilde{\mathcal{T}}$ of the triangulation $\mathcal{T}$ because the endpoints correspond to punctures.
 
  Up to conjugation by a quasiconformal map, we can assume that $\mathcal{T}$ has all shears zero. Therefore, $\mathcal{T}$ lifts to the Farey triangulation $\mathcal{F}$ of $\mathbb{H}^2$ and $\tilde{f}$ preserves the endpoints of $\mathcal{F}$. By \cite[Theorem 1.2]{PS}, the homeomorphism $\tilde{f}:\hat{\mathbb{R}}\to\hat{\mathbb{R}}$ is quasisymmetric if and only if $\mathcal{F}$ and $\tilde{f}(\mathcal{F})$ have finite intersections. If a conjugate of a homeomorphism by a quasisymmetric map is quasisymmetric then the homeomorphism is also quasisymmetric. The theorem follows.
 \end{proof}

\end{document}